\documentclass[preprint,11pt]{article}

\RequirePackage[colorlinks,citecolor=blue,urlcolor=blue]{hyperref}
\usepackage{amssymb}
\usepackage{mathrsfs}
\usepackage{graphicx}
\usepackage{colortbl,dcolumn}
\usepackage{tabularx}
\usepackage{amsmath}
\usepackage{psfrag}
\usepackage{booktabs}
\usepackage{cite}
\usepackage{amsthm}
\usepackage{float}
\usepackage{subfigure}
   
\numberwithin{equation}{section}
\usepackage[top=1.2in, bottom=1.2in, left=1in, right=1in]{geometry}

\newcommand{\C}{\mathbb{C}}
\newcommand{\R}{\mathbb{R}}
\newcommand{\N}{\mathbb{N}}
\newcommand{\E}{\mathbb{E}}
\renewcommand{\P}{\mathbb{P}}

\newcommand{\bi}{\mathbf{i}}
\newtheorem{tm}{Theorem}[section]
\newtheorem{df}{Definition}
\newtheorem{lm}{Lemma}
\newtheorem{prop}{Proposition}[section]

\newtheorem{rk}{Remark}

\allowdisplaybreaks \allowdisplaybreaks[4]%change line

\begin{document}

%\title{Ergodic Semi-discretization and Symplectic Full Discretization of Stochastic Schr\"odinger Equation with Multiplicative Noise}
%\title{\bf Ergodic Approximation of Stochastic NLS Equation with Multiplicative Noise via Structure-preserving Schemes }

\title{\bf Numerical Analysis on Ergodic Limit of Approximations for Stochastic NLS Equation via Multi-symplectic Scheme}

       \author{
        { Jialin Hong\footnotemark[1], Xu Wang\footnotemark[2], and Liying Zhang\footnotemark[3]}\\
      \\
        {\rm\small \footnotemark[1]~\footnotemark[2] Institute of Computational Mathematics and Scientific/Engineering Computing,}\\{\rm\small Academy of Mathematics and Systems Science, Chinese Academy of Sciences,}\\ 
        {\rm\small Beijing 100190, P.R.China. }\\
      {\rm\small \footnotemark[3] Department of Mathematics, College of Sciences,  China University of Mining and Technology, }\\
{\rm\small Beijing 100083, P.R.China. }
      }
       \maketitle
       \footnotetext{Authors are supported by National Natural Science Foundation of China (NO. 91530118, NO. 91130003, NO. 11021101 and NO. 11290142). }%(hjl@lsec.cc.ac.cn, wangxu@lsec.cc.ac.cn, lyzhang@lsec.cc.ac.cn)}
    %   \footnotetext{\footnotemark[3]Author is supported by {\color{red}foundation}. (lyzhang@lsec.cc.ac.cn)}
        \footnotetext{\footnotemark[2]Corresponding author: wangxu@lsec.cc.ac.cn.}

       \begin{abstract}
       We consider a finite dimensional approximation of the stochastic nonlinear Schr\"odinger equation driven by multiplicative noise, which is derived by applying a symplectic method to the original equation in spatial direction. Both the unique ergodicity and the charge conservation law for this finite dimensional approximation are obtained on the unit sphere. To simulate the ergodic limit over long time for the finite dimensional approximation, we discretize it further in temporal direction to obtain a fully discrete scheme, which inherits not only the stochastic multi-symplecticity and charge conservation law of the original equation but also the unique ergodicity of the finite dimensional approximation. The temporal average of the fully discrete numerical solution is proved to converge to the ergodic limit with order one with respect to the time step for a fixed spatial step. Numerical experiments verify our theoretical results on charge conservation, ergodicity and weak convergence.
         \\


\end{abstract}

\section{Introduction}

For the stochastic nonlinear Schr\"odinger (NLS) equation with a multiplicative noise in Stratonovich sense 
\begin{equation}\label{model}
\left\{
\begin{aligned}
&du=\mathbf{i}\big{(}\Delta u+\lambda|u|^2u\big{)}dt+\mathbf{i}u\circ dW,\\
&u(t,0)=u(t,1)=0,~t\geq0,\\
&u(0,x)=u_0(x),~x\in[0,1]
\end{aligned}
\right.
\end{equation}
with $\lambda=\pm 1$, we consider the case that $W$ is a real valued $Q$-Wiener process on a filtered probability space $(\Omega,\mathcal{F},(\mathcal{F}_t)_{t\geq0},\P)$ with paths in $H^1_0:=H_0^1(0,1)$ with Dirichlet boundary condition. The Karhunen--Lo\`eve expansion of $W$ is as follows
\begin{equation*}
W(t,x,\omega)=\sum_{k=0}^{\infty}\beta_k(t,\omega)Q^{\frac12}e_k(x),\quad t\geq0,\quad x\in[0,1],\quad\omega\in\Omega,
\end{equation*}
where $(e_k=\sqrt{2}\sin(k\pi x))_{k\geq1}$ is an eigenbasis of the Dirichlet Laplacian $\Delta$ in $L^2:=L^2(0,1)$ and $(\beta_k)_{k\geq1}$ is a sequence of independent real valued Brownian motions associated to the filtration $(\mathcal{F}_t)_{t\geq0}$. In addition, the covariance operator $Q$ is assumed to commute with the Laplacian and satisfies
$$Qe_k=\eta_ke_k,\quad\eta_k>0,\quad\forall k\in\mathbb{N},\quad\eta:=\sum_{k=1}^{\infty}\eta_k<\infty.$$
We refer to \cite{BD03} for additional assumptions on the well-posedness of \eqref{model}.
It is shown that \eqref{model} is a Hamiltonian system with stochastic multi-symplectic structure and charge conservation law (see \cite{chuchu2,BD03,JSS13} and references therein). 
Structure-preserving numerical schemes have remarkable superiority to conventional schemes on numerically solving Hamiltonian systems over long time.
As another kind of longtime behaviors, the ergodicity for this kind of conservative systems is an important and difficult problem which is still open. 
Motivated by \cite{EM01}, we study the ergodicity for a finite dimensional approximation (FDA) of the original equation instead.

In this paper, we investigate the ergodicity for a symplectic FDA of \eqref{model} and approximate its ergodic limit via a multi-symplectic and ergodic scheme. 
As we show that the FDA is charge conserved, without loss of generality, we consider the ergodicity in the finite dimensional unit sphere $\mathcal{S}$.
There have been some papers considering the additive noise case with dissipative assumptions, and also some papers requiring the uniformly elliptic assumption on the whole space to ensure the unique ergodicity (see e.g. \cite{brehier,mattingly,MT10,talay,Talay02}). 
 For the conservative FDA with a linear multiplicative noise, 
 it has an uncertain nondegeneracy, which relies heavily on the solution.
To overcome this difficulty, we construct an invariant control set $\mathcal{M}_0\subset\mathcal{S}$, in which the FDA is shown to be nondegenerate. 
Together with the Krylov--Bogoliubov theorem and the H\"ormander condition, we prove that the solution $U$ possesses a unique invariant measure $\mu_h$ (i.e., $U$ is uniquely ergodic) with 
 \begin{align*}
\lim_{T\to\infty}\frac1T\int_0^T\mathbb{E}f(U(t))dt=
\int_{\mathcal{M}_0}fd\mu_h=
\int_{\mathcal{S}}fd\mu_h.
\end{align*}

For many physical applications, the approximation of the invariant measure is of fundamental importance, especially when the invariant measure is unknown (see e.g. \cite{abdulle,brehier,brehier2,brehier3,WX15,MT10,MT07,talay,Talay02}). 
Some papers construct numerical schemes which also possess unique invariant measures, and then show the approximate error between invariant measures. For example, \cite{WX15,talay}  work with dissipative systems driven by additive noise, and \cite{Talay02} considers elliptic SDEs with bounded coefficients and dissipative type condition. 
There is also some work concentrating on the approximation of the invariant measure, i.e., the approximation of the ergodic limit $\int_{\mathcal{S}}fd\mu_h$, in which case the numerical schemes may not be uniquely ergodic. 
For instance, \cite{brehier} approximates the invariant measure of stochastic partial differential equations with an additive noise based on Kolmogorov equation. 
\cite{MT10} gives error estimates for time-averaging estimators of numerical schemes based on the associated Poisson equation and the assumption of local weak convergence order. 
Authors in \cite{MT07} calculate the ergodic limit for Langevin equations with dissipations via quasi-symplectic integrators. There has been few results on constructing conservative and uniquely ergodic schemes to calculate the ergodic limit for conservative systems to our knowledge. 
We focus on the approximation of the ergodic limit via a multi-symplectic scheme, which is also shown to be uniquely ergodic.
For a fixed spacial dimension, the local weak error of this fully discrete scheme (FDS) in temporal direction is of order two, which yields order one for the approximate error of the ergodic limit based on the associated Poisson equation (see also \cite{brehier2,MT10}) and a priori estimates of the numerical solutions. That is, 
\begin{align*}
\left|\mathbb{E}\left[\frac1N\sum_{n=0}^{N-1}f(U^n)-\int_{\mathcal{S}}fd\mu_h\right]\right|
\le C_h(\frac1T+\tau).
\end{align*}

The paper is organized as follows. In Section 2, we apply a symplectic semi-discrete scheme to the original equation to get the FDA, and show the unique ergodicity as well as  the charge conservation law for the FDA. In Section 3, we present a multi-symplectic and ergodic FDS to approximate the ergodic limit, and show the approximate error based on a priori estimates and local weak error. 
In Section 4, the discrete charge evolution compared with those of Euler--Maruyama scheme and implicit Euler scheme, ergodic limit and global weak convergence order are tested numerically. Section 5 is the appendix containing proofs of some a priori estimates.

\section{\textsc{Unique ergodicity}}
In this section, we first apply the central finite difference scheme to \eqref{model} in spatial direction to obtain a FDA, which is also a Hamiltonian system. To investigate the ergodicity of this conservative system, we then construct an invariant control set $\mathcal{M}_0\subset\mathcal{S}$ with respect to a control function introduced in Section \ref{ue}. The FDA is proved to be ergodic in $\mathcal{M}_0$ based on the Krylov--Bogoliubov theorem and the H\"ormander condition.

\subsection{Finite dimensional approximation (FDA)}
Based on the central finite difference scheme and the notation $u_j:=u_j(t)$, $j=1,\cdots,M$, we consider the following spatial semi-discretization
\begin{align*}
du_j=\mathbf{i}\left[\frac{u_{j+1}-2u_j+u_{j-1}}{h^2}+\lambda|u_j|^2u_j\right]dt+\mathbf{i}u_j\sum_{k=1}^K\sqrt{\eta_k}e_k(x_j)\circ d\beta_k(t)
\end{align*}
with a truncated noise $\sum_{k=1}^K\sqrt{\eta_k}e_k(x)\beta_k(t)$, $K\in\mathbb{N}$, a given uniform step size $h=\frac1{M+1}$ for some $M\le K$ and $x_j=jh$, $j=1,\cdots,M$. The condition $M\le K$ here ensures the existence of the solution for the control function. Denoting vectors $U:=U(t)=(u_1,\cdots,u_M)^T\in\C^M$, $\beta(t)=(\beta_1(t),\cdots,\beta_K(t))^T\in\R^K,$ and matrices $F(U)=diag\{|u_1|^2,\cdots,|u_M|^2\},$ $E_k= diag \{e_k(x_1),\cdots,e_k(x_M)\},$ $\Lambda= diag \{\sqrt{\eta_1},\cdots,\sqrt{\eta_K^{}}\}$, $Z(U)= diag \{u_1,\cdots,u_M\}E_{MK}\Lambda$, 
\begin{equation*}A=
\left(
\begin{array}{cccc}
-2& 1&  & \\
 1&-2& 1 & \\
  &\ddots&\ddots&\ddots \\
 & &1&-2
\end{array}
\right)\in\R^{M\times M},\;\;
E_{MK}=\begin{pmatrix}
    e_1(x_1)  &\cdots  & e_K(x_1) \\
     \vdots& &\vdots \\
     e_1(x_M) &\cdots & e_K(x_M)
\end{pmatrix}_{M\times K},
\end{equation*}
then the FDA is in the following form
\begin{equation}\label{spatial scheme}
\left\{
\begin{aligned}
&dU=\mathbf{i}\left[\frac1{h^2}AU+\lambda F(U)U\right]dt+\mathbf{i}Z(U)\circ d\beta(t),\\
&U(0)=c_*\left(u_0(x_1),\cdots,u_0(x_M)\right)^T,
\end{aligned}
\right.
\end{equation}
where $c_*$ is a normalized constant. The noise term in \eqref{spatial scheme} has an equivalent It\^o form
\begin{align}
\mathbf{i}Z(U)\circ d\beta(t)
=&\mathbf{i}\sum_{k=1}^K\sqrt{\eta_k}E_kU\circ d\beta_k(t)\nonumber
=-\frac12\sum_{k=1}^K\eta_kE_k^2Udt+\mathbf{i}\sum_{k=1}^K\sqrt{\eta_k}E_kUd\beta_k(t)\nonumber\\\label{noise}
=:&-\hat{E}Udt+\mathbf{i}\sum_{k=1}^K\sqrt{\eta_k}E_kUd\beta_k(t)
\end{align}
with $\hat{E}=\frac12\sum_{k=1}^K\eta_kE_k^2$.
In the sequel, $\|\cdot\|$ denotes the $2$-norm for both matrices and vectors, which satisfies
$\|BV\|\le\|B\|\|V\|$ 
for any matrices $B\in\C^{m\times n}$ and vectors $V\in\C^{n}$, $m,n\in\N$. It is then easy to show that $\|A\|\le4$, which is independent of the dimension $M$.

\begin{prop}\label{discreteju}
The FDA (\ref{spatial scheme}) possesses the charge conservation law, i.e.,
$$\|U(t)\|^2=\|U(0)\|^2,\quad\forall~t\geq0,\quad \mathbb{P}\text{-}a.s.$$
where $\|U(t)\|=(\|P(t)\|^2+\|Q(t)\|^2)^{\frac12}=\big{(}\sum_{m=1}^M(|p_m(t)|^2+|q_m(t)|^2)\big{)}^{\frac12}$, $P(t)=(p_1(t),\cdots,p_M(t))^T$ and $Q(t)=(q_1(t),\cdots,q_M(t))^T$ are the real and imaginary parts of $U(t)$ respectively.
\end{prop}
\begin{proof}
Noticing that matrices $A$ and $F(U)$ are symmetric  and the linear function $Z(U)$ satisfies
\begin{align}
\overline{U}^TZ(U)&=(\overline{u_1},\cdots,\overline{u_M})\begin{pmatrix}
    u_1  &   &   \\
      &\ddots &  \\
      &  & u_M
\end{pmatrix}E_{MK}\begin{pmatrix}
    \sqrt{\eta_1}  &   &   \\
      &\ddots &  \\
      &  & \sqrt{\eta_K^{}}
\end{pmatrix}\nonumber\\\label{z}
&=(|u_1|^2,\cdots,|u_M|^2)E_{MK}\begin{pmatrix}
    \sqrt{\eta_1}  &   &   \\
      &\ddots &  \\
      &  & \sqrt{\eta_K^{}}
\end{pmatrix}\in\R^K,
\end{align}
where $\overline{U}$ denotes the conjugate of $U$, we multiply \eqref{spatial scheme} by $\overline{U}^T$, take the real part, and then get the charge conservation law for $U$.
\end{proof}

In the sequel, without pointing out, all equations hold in the sense $\P$-a.s.
\begin{rk}
Eq. \eqref{model} can be rewritten into an infinite dimensional Hamiltonian system (see \cite{JSS13}).
It is easy to verify that the central finite difference scheme \eqref{spatial scheme} applied to \eqref{model} is equivalent to the symplectic Euler scheme applied to the infinite dimensional Hamiltonian form of \eqref{model}, which implies the symplecticity of \eqref{spatial scheme}.
\end{rk}

\subsection{Unique ergodicity}\label{ue}
As the charge of \eqref{spatial scheme} is conserved shown in  Proposition \ref{discreteju}, without loss of generality,  we assume that $U(0)\in\mathcal{S}$ and investigate the unique ergodicity of (\ref{spatial scheme}) on $\mathcal{S}$. As the nondegeneracy for \eqref{spatial scheme} relies on the solution $U$ as a result of the multiplicative noise, the standard procedure to show the irreducibility and strong Feller property on the whole $\mathcal{S}$ do not apply.
So we need to construct an invariant control set. 

\begin{df}(see e.g. \cite{AK87})
A subset $\mathcal{M}\neq\emptyset$ of $\mathcal{S}$ is called an invariant control set for the control system
\begin{equation}\label{control}
d\phi=\mathbf{i}\left[\frac1{h^2}A\phi+\lambda F(\phi)\phi\right]dt+\mathbf{i}Z(\phi) d\Psi(t)
\end{equation}
of (\ref{spatial scheme}) with a differentiable deterministic function $\Psi$, if 
$\,\overline{\mathcal{O}^+(x)}=\overline{\mathcal{M}},\;\forall x\in\mathcal{M},$
and $\mathcal{M}$ is maximal with respect to inclusion, where $\mathcal{O}^+(x)$ denotes the set of points reachable from $x$ (i.e., connected with $x$) in any finite time and $\overline{\mathcal{M}}$ denotes the closure of $\mathcal{M}$.
\end{df}

We state one of our main results in the following theorem.

\begin{tm}\label{ergo1}
The FDA (\ref{spatial scheme}) possesses a unique invariant probability measure $\mu_h$ on an invariant control set $\mathcal{M}_0$, which implies the unique ergodicity of (\ref{spatial scheme}). 
Moreover, $$supp(\mu_h)=\mathcal{S}\text{ and }\mu_h(\mathcal{S})=\mu_h(\mathcal{M}_0)=1.$$
\end{tm}
\begin{proof}
\textbf{Step 1. Existence of invariant measures.} 

From Proposition \ref{discreteju}, we find 
$\pi_t(U(0),\mathcal{S})=1,\forall\,t\ge0,$
where $\pi_t(U(0),\cdot)$ denotes the transition probability (probability kernel) of $U(t)$.
As the finite dimensional unit sphere $\mathcal{S}$ is tight, 
the family of measures $\pi_t(U(0),\cdot)$ is tight,
which implies the existence of invariant measures by the Krylov--Bogoliubov theorem \cite{daprato}.

\textbf{Step 2. Invariant control set.} 

Denoting $U=P+\mathbf{i}Q$ with $P$ and $Q$ being the real and imaginary parts of $U$ respectively, we first consider the following subset of $\mathcal{S}$
$$\mathcal{S}_1=\{U=P+\mathbf{i}Q\in \mathcal{S}:\;P>0\}.$$
For any $t>0,\;y,z\in\mathcal{S}_1$, there exists a differentiable function $\phi$ satisfying $\phi(s)=(\phi_1(s),\cdots,\phi_M(s))^T\in\mathcal{S}_1$, $s\in[0,t]$, $\phi(0)=y$ and
$\phi(t)=z$
by polynomial interpolation argument. 
As rank$\left(Z(\phi(s))\right)=M$ for $\phi(s)\in\mathcal{S}_1$ and $M\le K$,
the linear equations
$$Z(\phi(s))X=-\bi\phi'(s)-\left[\frac1{h^2}A\phi(s)+\lambda F(\phi(s))\phi(s)\right]$$
possess a solution $X\in\C^M$. As in addition $Z(\phi(s))=diag\{\phi_1(s),\cdots,\phi_M(s)\}E_{MK}\Lambda$, where $diag\{\phi_1(s),\cdots,\phi_M(s)\}$ is invertible for $\phi(s)\in\mathcal{S}_1$, the solution $X$ depends continuously on $s$ and is denoted by $X(s)$. Thus, there exists a differentiable function $\Psi(\cdot):=\int_0^{\cdot}X(s)ds$  which, together with $\phi$ defined above, satisfies the control function \eqref{control} with initial data $\Psi(0)=0$.
That is, for any $y,z\in\mathcal{S}_1$, $y$ and $z$ are connected, denoted by $y\leftrightarrow z$. 
The above argument also holds for the following subsets
\begin{align*}
\mathcal{S}_2=&\{U=P+\mathbf{i}Q\in \mathcal{S}:\;P<0\},\\
\mathcal{S}_3=&\{U=P+\mathbf{i}Q\in \mathcal{S}:\;Q>0\},\\
\mathcal{S}_4=&\{U=P+\mathbf{i}Q\in \mathcal{S}:\;Q<0\}.
\end{align*}
For any $y\in\mathcal{S}_i,\;z\in\mathcal{S}_j$ with $i\neq j$ and $i,j\in\{1,2,3,4\}$, there must exist $\mathcal{S}_l$, $r_i$ and $r_j$,  satisfying $r_i\in\mathcal{S}_i\cap\mathcal{S}_l\neq\emptyset$ and $r_j\in\mathcal{S}_j\cap\mathcal{S}_l\neq\emptyset$ for some $l\in\{1,2,3,4\}$,  such that $y\leftrightarrow r_i\leftrightarrow r_j\leftrightarrow z$.  
Thus, $$\mathcal{M}_0:=\mathcal{S}_1\cup\mathcal{S}_2\cup\mathcal{S}_3\cup\mathcal{S}_4=\{U=P+\mathbf{i}Q\in\mathcal{S}: P\neq0\;\text{or}\;Q\neq0\},$$
with $\overline{\mathcal{M}_0}=\mathcal{S}$,
is an invariant control set for (\ref{control}).

\textbf{Step 3. Uniqueness of the invariant measure.} 

We rewrite (\ref{spatial scheme}) with $P$ and $Q$ according to its equivalent form in It\^o sense and obtain
\begin{align}\label{2ito}
d\begin{pmatrix}
    P \\
   Q
\end{pmatrix}=&\begin{pmatrix}
   -\hat{E}  & -\frac1{h^2}A-\lambda F(P,Q)   \\
    \frac1{h^2}A+\lambda F(P,Q) & -\hat{E}
\end{pmatrix}\begin{pmatrix}
    P \\
   Q
\end{pmatrix}dt\nonumber\\&+\sum_{k=1}^K\sqrt{\eta_k}\begin{pmatrix}
    0&-E_k \\
   E_k&0
\end{pmatrix}\begin{pmatrix}
    P \\
   Q
\end{pmatrix} d\beta_k(t)\nonumber\\
=:&X_0(P,Q)dt+\sum_{k=1}^KX_k(P,Q)d\beta_k(t).
\end{align}
 To derive the uniqueness of the invariant measure, we consider the Lie algebra generated by the diffusions of (\ref{2ito})
$$L(X_0,X_1,\cdots,X_K)=span\bigg{\{}X_l,[X_i,X_j],\left[X_l,[X_i,X_j]\right],\cdots,0\leq l,i,j\leq K\bigg{\}}.$$
Choosing $p_*=0$ and $q_*=\frac{-1}{\sqrt{M}}(1,\cdots,1)^T$ such that $z_*:=p_*+\mathbf{i}q_*\in\mathcal{S}_4\subset\mathcal{M}_0$, we derive that the following vectors 
\begin{align*}
X_k(p_*,q_*)=\sqrt{\frac{\eta_k}M}\begin{pmatrix}
    e_k(x_1) \\
    \vdots\\
   e_k(x_M)\\
   0\\
   \vdots\\
   0
\end{pmatrix},\;[X_0,X_k](p_*,q_*)=\sqrt{\frac{\eta_k}M}\begin{pmatrix}
   -\hat{E}\begin{pmatrix} e_k(x_1) \\
    \vdots\\
   e_k(x_M)\end{pmatrix}\\
(\frac1{h^2}A+\frac1MI)\begin{pmatrix}  e_k(x_1) \\
    \vdots\\
   e_k(x_M)\end{pmatrix}
\end{pmatrix}
\end{align*}
are independent of each other for $k=1,\cdots,M$, which hence implies the following H\"ormander condition
$$\text{dim}\,L(X_0,X_1,\cdots,X_K)(z_*)=2M.$$
Then there is at most one invariant measure with $supp(\mu_h)=\mathcal{S}$ according to \cite{AK87}.
Actually, according to above procedure, we obtain that H\"ormander condition holds uniformly for any $z\in\mathcal{M}_0$.

Combining the three steps above, we conclude that there exists a unique invariant measure $\mu_h$ on $\mathcal{M}_0$ for the FDA, with $\mu_h(\mathcal{S})=\mu_h(\mathcal{M}_0)=1$.
\end{proof}

From the theorem above, we can find out that for some other nonlinearities, e.g. $iF(x,|u|)u$ with $F$ being some potential function, such that the equation still possesses the charge conservation law, we can still get the ergodicity of the finite dimensional approximation of the original equation through the   procedure above. The procedure could also applied to higher dimensional Schr\"odinger equations with proper well-posed assumptions, but it may be more technical to verify the H\"ormander condition.
\begin{rk}
According to the ergodicity of (\ref{spatial scheme}), we have
\begin{align*}
\lim_{T\to\infty}\frac1T\int_0^T\mathbb{E}f(U(t))dt=\int_{\mathcal{S}}fd\mu_h,\quad\forall~f\in B_b(\mathcal{S}),\quad\text{in}\;\;L^2(\mathcal{S},\mu_h),
\end{align*}
where $B_b(\mathcal{S})$ denotes the set of bounded and measurable functions and $\int_{\mathcal{S}}fd\mu_h$ is known as the ergodic limit with respect to the invariant measure $\mu_h$. 

For more details, we refer to \cite{daprato} and references therein.
\end{rk}

\section{Approximation of ergodic limit}

A fully discrete scheme (FDS) with the discrete multi-symplectic structure and the discrete charge conservation law is constructed in this section, which could also inherit the unique ergodicity of the FDA. In addition, we prove that the time average of the FDS can approximate the ergodic limit $\int_{\mathcal{S}}fd\mu_h$ with order one with respect to the time step.

\subsection{Fully discrete scheme (FDS)}
We apply the midpoint scheme to (\ref{spatial scheme}), and obtain the following FDS
\begin{equation}\label{full}
\left\{
\begin{aligned}
&U^{n+1}-U^n=\mathbf{i}\frac{\tau}{h^2}AU^{n+\frac12}
+\bi\lambda\tau F(U^{n+\frac12})U^{n+\frac12}
+\mathbf{i}Z(U^{n+\frac12})\delta_{n+1}\beta,\\
&U^0=U(0)\in\mathcal{S},
\end{aligned}
\right.
\end{equation}
where $\tau$ denotes the uniform time step, $t_n=n\tau$, $U^n=(u_1^n,\cdots,u_M^n)\in\C^M$, $U^{n+\frac12}=\frac{U^{n+1}+U^n}2$ and $\delta_{n+1}\beta=\beta(t_{n+1})-\beta(t_n)$.  
For the FDS \eqref{full}, which is implicit in both deterministic and stochastic terms, 
its well-posedness is stated in the following proposition. 

\begin{prop}\label{ju}
For any initial value $U^0=U(0)\in\mathcal{S}$, there exists a unique solution $(U^{n})_{n\in\N}$ of \eqref{full}, and it possesses the discrete charge conservation law, i.e.,
\begin{align*}
\|U^{n+1}\|^2=\|U^n\|^2=1,\quad\forall~ n\in\N.
\end{align*}
\end{prop}
\begin{proof}
We multiply both sides of \eqref{full} by $\overline{U^{n+\frac12}}$, take the real part, and obtain the existence of the numerical solution by the Brouwer fixed-point theorem as well as the discrete charge conservation law.

For the uniqueness, we assume that $X=(X_1,\cdots,X_M)^T$ and $Y=(Y_1,\cdots,Y_M)^T$ are two solutions of \eqref{full} with $U^n=z=(z_1,\cdots,z_M)^T\in\mathcal{S}$. It follows that $X,Y\in\mathcal{S}$ and
\begin{equation}\label{x-y}
X-Y=\mathbf{i} \frac{\tau}{h^2}A\frac{X-Y}2+\frac{\mathbf{i}\lambda\tau}8 H(X,Y,z)+\mathbf{i}Z(\frac{X-Y}2)\delta_{n+1}\beta,
\end{equation}
where
\begin{equation*}H(X,Y,z)=
\begin{pmatrix}
|X_1+z_1|^2(X_1+z_1)-|Y_1+z_1|^2(Y_1+z_1)\\
\vdots\\
|X_M+z_M|^2(X_M+z_M)-|Y_M+z_M|^2(Y_M+z_M)
\end{pmatrix}.
\end{equation*}
Based on the fact that $|a|^2a-|b|^2b=|a|^2(a-b)+|b|^2(a-b)+ab(\overline{a}-\overline{b})$ for any $a,b\in\C$, we have
\begin{align*}
\Im\left[(\overline{X}-\overline{Y})^TH(X,Y,z)\right]=\Im\left[\sum_{m=1}^M(X_m+z_m)(Y_m+z_m)(\overline{X_m}-\overline{Y_m})^2\right]
\end{align*}
with $\Im[V]$ denoting the imaginary part of $V$. 
Multiplying \eqref{x-y} by $(\overline{X}-\overline{Y})^T$, taking the real part, and we get
\begin{align*}
&\|X-Y\|^2=-\frac{\lambda\tau}8\Im\left[(\overline{X}-\overline{Y})^TH(X,Y,z)\right]\\
\leq&\frac\tau8\left(\max_{1\le m\le M}|X_m+z_m||Y_m+z_m|\right)\|X-Y\|^2
\le\frac{\tau}2\|X-Y\|^2,
\end{align*}
where we have used the fact $X,Y,z\in\mathcal{S}$ and \eqref{z}.
For $\tau<1$, we get $X=Y$ and complete the proof.
\end{proof}

The proposition above shows that \eqref{full} possesses the discrete charge conservation law. Furthermore, \eqref{full} also inherits the unique ergodicity of the FDA and the stochastic multi-symplecticity of the original equation, which are stated in the following two theorems.
\begin{tm}\label{ergo2}
The FDS \eqref{full} is also ergodic with a unique invariant measure $\mu_h^{\tau}$ on the control set $\mathcal{M}_0$, such that $\mu_h^{\tau}(\mathcal{S})=\mu_h^{\tau}(\mathcal{M}_0)=1$. Also,
\begin{align*}
\lim_{N\rightarrow\infty}\frac1{N}\sum_{n=0}^{N-1}f(U^n)=\int_{\mathcal{S}}fd\mu_h^{\tau},\quad\forall~f\in B_b(\mathcal{S}),\quad\text{in}\;\;L^2(\mathcal{S},\mu_h^{\tau}).
\end{align*}
\end{tm}
\begin{proof}
Based on the charge conservation law for $\{U^n\}_{n\ge1}$, we obtain the existence of the invariant measure similar to the proof of Theorem \ref{ergo1}. 

To obtain the uniqueness of the invariant measure, we show that the Markov chain $\{U^{3n}\}_{n\ge1}$ satisfies the minorization condition (see e.g. \cite{mattingly}). 
Firstly, Proposition \ref{ju} implies that for a given $U^n\in\mathcal{S}$, solution $U^{n+1}$ can be defined through a continuous function $U^{n+1}=\kappa(U^n,\delta_{n+1}\beta)$.
As $\delta_{n+1}\beta$ has a $C^{\infty}$ density, we get a jointly continuous density for $U^{n+1}$.
Secondly, similar to Theorem \ref{ergo1}, for any given $y,z\in\mathcal{M}_0$, there must exist $i,j,k\in\{1,2,3,4\}$ and $r_i,r_j\in\mathcal{M}_0$, such that $y\in\mathcal{S}_i,z\in\mathcal{S}_j,r_i\in\mathcal{S}_i\cap\mathcal{S}_k$ and $r_j\in\mathcal{S}_j\cap\mathcal{S}_k$. As $\frac{y+r_i}2\in\mathcal{S}_i$ and $Z(\frac{y+r_i}2)$ is invertible, $\delta_{3n+1}\beta$ can be chosen to ensure that 
\begin{align*}
r_i-y=\mathbf{i}\frac{\tau}{h^2}A\frac{y+r_i}2
+\bi\lambda\tau F(\frac{y+r_i}2)\frac{y+r_i}2
+\mathbf{i}Z(\frac{y+r_i}2)\delta_{3n+1}\beta
\end{align*}  
holds, i.e., $r_i=\kappa(y,\delta_{3n+1}\beta)$. Similarly, based on the fact $\frac{r_i+r_j}2\in\mathcal{S}_k$ and $\frac{r_j+z}2\in\mathcal{S}_j$, we have $r_j=\kappa(r_i,\delta_{3n+2}\beta)$ and $z=\kappa(r_j,\delta_{3n+3}\beta)$. That is, for any given $y,z\in\mathcal{M}_0$, $\delta_{3n+1}\beta,\delta_{3n+2}\beta,\delta_{3n+3}\beta$ can be chosen to ensure that $U^{3n}=y$ and $U^{3(n+1)}=z$. Finally we obtain that, for any $\delta>0$, 
\begin{align*}
\P_{3}\left(y,B(z,\delta)\right):=\P\left(U^3\in B(z,\delta)\big{|}U^0=y\right)>0,
\end{align*}
where $B(z,\delta)$ denotes the open ball of radius $\delta$ centered at $z$.
\end{proof}

The infinite dimensional system \eqref{model} has been shown to preserve the stochastic multi-symplectic conservation law locally (see i.e. \cite{JSS13})
\begin{equation*}
d_t(dp\wedge dq)-\partial_x(dp\wedge dv+dq\wedge dw)dt=0
\end{equation*}
with $p,q$ denoting the real and imaginary parts of solution $u$ respectively and $v=p_x$, $w=q_x$ being the derivatives of $p$ and $q$ with respect to variable $x$. 
We now show that this ergodic FDS \eqref{full} not only possesses the discrete charge conservation law as shown in Proposition \ref{ju} but also preserves the discrete stochastic multi-symplectic structure.

\begin{tm}
The implicit FDS \eqref{full} preserves the discrete multi-symplectic structure
\begin{align*}
&\frac1\tau(dp_j^{n+1}\wedge dq_j^{n+1}-dp_j^{n}\wedge dq_j^{n})-\frac1h(dp_j^{n+\frac12}\wedge dv_{j+1}^{n+\frac12}-dp_{j-1}^{n+\frac12}\wedge dv_j^{n+\frac12})\\
&-\frac1h(dq_j^{n+\frac12}\wedge dw_{j+1}^{n+\frac12}-dq_{j-1}^{n+\frac12}\wedge dw_j^{n+\frac12})=0,
\end{align*}
where $p_j^n,q_j^n$ denote the real and imaginary parts of $u_j^n$, $v_j=\frac1h(p_j^n-p_{j-1}^n)$ and $w_j=\frac1h(q_j^n-q_{j-1}^n)$.
\end{tm}
\begin{proof}
Rewriting \eqref{full} with the real and imaginary parts of the components $u_j^n$ of $U^n$, we get
\begin{equation}\label{sym}
\left\{
\begin{aligned}
\frac1\tau(q_j^{n+1}-q_j^n)-\frac1h(v_{j+1}^{n+\frac12}-v_j^{n+\frac12})&=
\left((p_j^{n+\frac12})^2+(q_j^{n+\frac12})^2\right)p_j^{n+\frac12}+p_j^{n+\frac12}\zeta_j^K,\\
-\frac1\tau(p_j^{n+1}-p_j^n)-\frac1h(w_{j+1}^{n+\frac12}-w_j^{n+\frac12})&=
\left((p_j^{n+\frac12})^2+(q_j^{n+\frac12})^2\right)q_j^{n+\frac12}+q_j^{n+\frac12}\zeta_j^K,\\
\frac1h(p_j^{n+\frac12}-p_{j-1}^{n+\frac12})&=v_j^{n+\frac12},\\
\frac1h(q_j^{n+\frac12}-q_{j-1}^{n+\frac12})&=w_j^{n+\frac12},
\end{aligned}
\right.
\end{equation}
where $\zeta_j^K=\sum_{k=1}^K\sqrt{\eta_k}e_k(x_j)\circ d\beta_k(t)$. 
Denoting $z_j^{n+\frac12}=(p_j^{n+\frac12},q_j^{n+\frac12},v_j^{n+\frac12},w_j^{n+\frac12})^T$ and taking differential in the phase space on both sides of \eqref{sym}, we obtain
\begin{align}
&\frac1\tau d\begin{pmatrix}
    q_j^{n+1}-q_j^n \\
    -(p_j^{n+1}-p_j^n)\\
   0\\
   0
\end{pmatrix}+\frac1h d\begin{pmatrix}
    -(v_{j+1}^{n+\frac12}-v_j^{n+\frac12}) \\
    -(w_{j+1}^{n+\frac12}-w_j^{n+\frac12})\\
   p_j^{n+\frac12}-p_{j-1}^{n+\frac12}\\
   q_j^{n+\frac12}-q_{j-1}^{n+\frac12}
\end{pmatrix}\nonumber\\\label{symplectic}
=&\nabla^2 S_1(z_j^{n+\frac12})dz_j^{n+\frac12}+\nabla^2 S_2(z_j^{n+\frac12})dz_j^{n+\frac12}\zeta_j^K,
\end{align}
where $$S_1(z_j^{n+\frac12})=\frac14\left((p_j^{n+\frac12})^2+(q_j^{n+\frac12})^2\right)^2+\frac12\left(v_j^{n+\frac12}\right)^2+\frac12\left(w_j^{n+\frac12}\right)^2$$
and $$S_2(z_j^{n+\frac12})=\frac12\left(p_j^{n+\frac12}\right)^2+\frac12\left(q_j^{n+\frac12}\right)^2.$$
Then the wedge product between $dz_j^{n+\frac12}$ and \eqref{symplectic} concludes the proof based on the symmetry of $\nabla^2 S_1$ and $\nabla^2S_2$.
\end{proof}

Before giving the approximate error of the ergodic limit, we give some essential a priori estimates about the stability of \eqref{full} and \eqref{spatial scheme}. In the following, $C$ denotes a generic constant independent of $T$, $N$, $\tau$ and $h$ while $C_h$ denotes a constant depending also on $h$, whose value may be different from line to line.

\begin{lm}\label{stable}
For any initial value $U^0\in\mathcal{S}$ and $\gamma\ge1$,  if $Q\in\mathcal{HS}(L^2,H^{\frac32-\frac1\gamma})$, then there exists a constant $C$ such that the solution $(U^n)_{n\in\N}$ of \eqref{full} satisfies
\begin{align*}
\E\left\|U^{n+1}-U^n\right\|^{2\gamma}\le C(\tau^{2\gamma}h^{-4\gamma}+\tau^\gamma),\quad\forall\;n\in\N,
\end{align*}
where $\mathcal{HS}(L^{\gamma_1},H^{\gamma_2})$ denotes the space of Hilbert--Schmidt operators from $L^{\gamma_1}$ to $H^{\gamma_2}$.
\end{lm}

\begin{lm}\label{stable2}
For any initial value $U(0)\in\mathcal{S}$ and $\gamma\ge1$, there exists a constant $C$ such that the solution $U(t)$ of \eqref{spatial scheme} satisfies
\begin{align*}
\E\|U(t_{n+1})-U(t_n)\|^{2\gamma}\le C({\tau}^{2\gamma}h^{-{4\gamma}}+\tau^\gamma),\quad\forall\;n\in\N.
\end{align*}  
\end{lm}

The proofs of Lemmas above are given in the appendix for readers' convenience.

\subsection{Approximation of ergodic limit}

To approximate the ergodic limit of \eqref{spatial scheme} and  get the approximate error, we give an estimate of the local weak convergence between $U(\tau)$ and $U^1$, and the Poisson equation associated to \eqref{spatial scheme} are also used (see \cite{MT10}).
Recall that the SDE \eqref{spatial scheme} in Stratonovich sense has an equivalent It\^o form 
\begin{align}
dU=&\left[\bi\frac{1}{h^2}AU+\mathbf{i}\lambda F(U)U-\hat EU\right]dt+\mathbf{i}Z(U)d\beta(t)\nonumber\\\label{ito}
=:&b(U)dt+\sigma(U)d\beta(t)
\end{align}
based on \eqref{noise}. 
For any fixed $f\in W^{4,\infty}(\mathcal{S})$, let $\hat{f}:=\int_{\mathcal{S}}fd\mu_h$
and $\varphi$ be the unique solution of the Poisson equation
$
\mathcal{L}\varphi=f-\hat f,
$
where 
\begin{align*}
\mathcal{L}:=b\cdot\nabla+\frac12\sigma\sigma^T:\nabla^2
\end{align*}
denotes the generator of \eqref{ito}. It's easy to find out that \eqref{ito} satisfies the hypoelliptic setting (see e.g. \cite{MT10}) according to the H\"ormander condition in Theorem \ref{ergo1}.  
Thus, $\varphi\in W^{4,\infty}(\mathcal{S})$ according to Theorem 4.1 in \cite{MT10}.  
Based on the well-posedness of the numerical solution $(U^n)_{n\in\N}$ and the implicit function theorem, \eqref{full} can be rewritten in the following form
\begin{align}\label{U}
U^{n+1}=U^n+\tau\Phi(U^n,\tau,h,\delta_{n+1}\beta)
\end{align}
for some function $\Phi$.  
Denoting by $D\varphi(u)\Phi_1$ and $D^k\varphi(u)(\Phi_1,\cdots,\Phi_k)$ the first and $k$-th order weak derivatives evaluated in the directions $\Phi_j$, $j=1,\cdots,k$ with $D^k\varphi(u)(\Phi)^k$ for short if all the directions are the same in the $k$-th derivatives, then we have
\begin{align}\label{Lphi}
\varphi(U^{n+1})=&\varphi(U^n)+\tau\left[D\varphi(U^n)\Phi^n+\frac12\tau D^2\varphi(U^n)(\Phi^n)^2\right]
+\frac16D^3\varphi(U^n)(\tau\Phi^n)^3+R^\Phi_n\nonumber\\
=:&\varphi(U^n)+\tau\mathcal{L}^\Phi\varphi(U^n)+\frac16D^3\varphi(U^n)(\tau\Phi^n)^3+R^\Phi_n,
\end{align}
where $\Phi^n:=\Phi(U^n,\tau,h,\delta_{n+1}\beta)$,
$$\mathcal{L}^\Phi\varphi(U^n)=D\varphi(U^n)\Phi^n+\frac12\tau D^2\varphi(U^n)(\Phi^n)^2$$ 
and 
$$R^\Phi_n=\frac1{4!}D^4\varphi(\theta_n)(\tau\Phi^n)^4$$ for some $\theta_n\in[U^n,U^{n+1}]:=[u_1^n,u_1^{n+1}]\times\cdots\times[u_M^n,u_M^{n+1}]$.
Adding \eqref{Lphi} together from $n=0$ to $n=N-1$ for some fixed $N\in\N$, then dividing the result by $T=N\tau$, and noticing that $\mathcal{L}\varphi(U^n)=f(U^n)-\hat f$, we obtain
\begin{align*}
\frac{\varphi(U^N)-\varphi(U^0)}{N\tau}
=&\frac1{N}\Bigg{(}\sum_{n=0}^{N-1}\big{[}\mathcal{L}^\Phi\varphi(U^n)-\mathcal{L}\varphi(U^n)\big{]}+\sum_{n=0}^{N-1}\mathcal{L}\varphi(U^n)\\
&+\frac1{\tau}\sum_{n=0}^{N-1}\frac16D^3\varphi(U^n)(\tau\Phi^n)^3+\frac1{\tau}\sum_{n=0}^{N-1}R^\Phi_n\Bigg{)}\\
=&\frac1{N}\sum_{n=0}^{N-1}\big{[}\mathcal{L}^\Phi\varphi(U^n)-\mathcal{L}\varphi(U^n)+\frac1{6\tau}D^3\varphi(U^n)(\tau\Phi^n)^3\big{]}\\
&+\left(\frac1N\sum_{n=0}^{N-1}f(U^n)-\hat f\right)+\frac1{N\tau}\sum_{n=0}^{N-1}R^\Phi_n,
\end{align*}
which shows
\begin{align}
&\left|\mathbb{E}\left[\frac1N\sum_{n=0}^{N-1}f(U^n)-\hat f\right]\right|
\leq\left|\frac1{N\tau}\E\left[\varphi(U^N)-\varphi(U^0)\right]\right|+\left|\frac1{N\tau}\sum_{n=0}^{N-1}\E R^\Phi_n\right|\nonumber\\\label{error}
+&\left|\frac1{N}\sum_{n=0}^{N-1}\E\left[\mathcal{L}^\Phi\varphi(U^n)-\mathcal{L}\varphi(U^n)
+\frac1{6\tau}D^3\varphi(U^n)(\tau\Phi^n)^3\right]\right|
=:\uppercase\expandafter{\romannumeral1}+\uppercase\expandafter{\romannumeral2}+\uppercase\expandafter{\romannumeral3}.
\end{align}
The average $\frac1N\sum_{n=0}^{N-1}f(U^n)$ is regard as an approximation of $\hat f$. We next begin to investigate the approximate error by estimating $\uppercase\expandafter{\romannumeral1}$, $\uppercase\expandafter{\romannumeral2}$ and $\uppercase\expandafter{\romannumeral3}$ respectively.

According to the fact that $\varphi\in W^{4,\infty}(\mathcal{S})$ and Lemma \ref{stable}, we have 
\begin{align}\label{term1}
\uppercase\expandafter{\romannumeral1}
\le\frac{2\|\varphi\|_{0,\infty}}{N\tau}\le \frac CT
\end{align}
and
\begin{align}
\uppercase\expandafter{\romannumeral2}
\leq&\frac1{N\tau}\sum_{n=0}^{N-1}\E\left[\left\|\tau\Phi^n\right\|^4\|D^4\varphi\|_{L^{\infty}}\right]
\le\frac C{N\tau}\sum_{n=0}^{N-1}\E\left[\left\|U^{n+1}-U^n\right\|^4\right]\nonumber\\\label{term2}
\le&\frac C{N\tau}\sum_{n=0}^{N-1}\left({\tau^4}{h^{-8}}+\tau^{2}\right)
\le C\left({\tau^3}{h^{-8}}+\tau\right),
\end{align}
where $\|\varphi\|_{\gamma,\infty}:=\sup_{|\alpha|\le \gamma,u\in\mathcal{S}}|D^{\alpha}\varphi(u)|$, $\gamma\in\N$.

It then remains to estimate the term $\uppercase\expandafter{\romannumeral3}$. To this end, we need  the  estimate about the local weak convergence, which is stated in the following theorem. The proof of the following theorem is also given in the appendix.

\begin{tm}\label{weakerror}
For a fixed spatial approximation \eqref{spatial scheme}, and for any initial value $U^0\in\mathcal{S}$ and $\varphi\in W^{4,\infty}(\mathcal{S})$, it holds under the condition $Q\in\mathcal{HS}(L^2,H^{\frac54})$ and $\tau=O(h^4)$ that
\begin{align*}
\left|\E\left[\varphi(U(\tau))-\varphi(U^1)\right]\right|\le C_h\tau^{2}
\end{align*}
for some constant $C_h=C(\varphi,\eta,h)$.
\end{tm}

Now we are in the position of showing the approximation error between the time average of FDS and the ergodic limit of FDA.

\begin{tm}\label{main}
Under the assumptions in Theorem \ref{weakerror} and for any $f\in W^{4,\infty}(\mathcal{S})$, there exists a positive constant $C_h=C(f,\eta,h)$ such that
\begin{align*}
\left|\mathbb{E}\left[\frac1N\sum_{n=0}^{N-1}f(U^n)-\hat f\right]\right|
\le C_h(\frac1T+\tau).
\end{align*}
\end{tm}
\begin{proof}
Based on \eqref{error}--\eqref{term2}, it suffices to estimate term $\uppercase\expandafter{\romannumeral3}$. 
For any $f\in W^{4,\infty}(\mathcal{S})$, we know from the statement above that the solution to the Poisson equation $ \mathcal{L}\varphi=f-\hat f$ satisfies $\varphi\in W^{4,\infty}(\mathcal{S})$. Based on \eqref{Lphi}, Lemma \ref{stable} and the condition $\tau=O(h^4)$, we have
\begin{align}\label{phi-u1}
\varphi(U^1)
\overset{\E}{=}&\varphi(U^0)+\tau\mathcal{L}^{\Phi}\varphi(U^0)+\frac16D^3\varphi(U^0)(U^1-U^0)^3+O(\tau^{2})\nonumber\\
\overset{\E}{=}&\varphi(U^0)+\tau\mathcal{L}^{\Phi}\varphi(U^0)+O(\tau^{2}),
\end{align}
where $\overset{\E}{=}$ means that the equation holds in expectation sense, and in the last step we have used the fact that
\begin{align}\label{varphiU}
D^3\varphi(U^0)(U^1-U^0)^3=&D^3\varphi(U^0)\left(\bi\frac{\tau}{h^2}AU^{\frac12}+\bi\lambda\tau F(U^{\frac12})U^{\frac12}+\bi Z(U^{\frac12})\delta_1\beta\right)^3\nonumber\\
\overset{\E}{=}&D^3\varphi(U^0)\left(\bi Z(U^{\frac12})\delta_1\beta\right)^3+O(\tau^2h^{-2}+\tau^2)\nonumber\\
\overset{\E}{=}&D^3\varphi(U^0)\left(\frac{\bi}2 Z(U^1-U^0)\delta_1\beta+\bi Z(U^0)\delta_1\beta\right)^3+O(\tau^2h^{-2}+\tau^2)\nonumber\\
\overset{\E}{=}&O(\tau^2h^{-2}+\tau^2)
\end{align}
based on the linearity of $Z$, Lemma \ref{stable} and that 
$\E\left(\bi Z(U^0)\delta_1\beta\right)^3=0$. 
We can also get the following expression similar to \eqref{phi-u1} based on Taylor expansion and Lemma \ref{stable2}
\begin{align}\label{phi-u-tau}
\varphi(U(\tau))\overset{\E}{=}&\varphi(U^0)+\int_0^\tau\left(D\varphi(U^0)b(U(t))+\frac12D^2\varphi(U^0)\left(\sigma(U(t))\right)^2\right)dt\nonumber\\
&+\int_0^\tau D\varphi(U^0)\sigma(U(t)) d\beta(t)
+\frac16D^3\varphi(U^0)(U(\tau)-U^0)^3+O(\tau^2)\nonumber\\
\overset{\E}{=}&\varphi(U^0)+\int_0^\tau\tilde{\mathcal{L}_t}\varphi(U^0)dt+O(\tau^2),
\end{align}
where 
$$\tilde{\mathcal{L}_t}\varphi(U^0):=D\varphi(U^0)b(U(t))+\frac12D^2\varphi(U^0)\left(\sigma(U(t))\right)^2$$ 
and 
$
\E\left[\int_0^\tau D\varphi(U^0)\sigma(U(t)) d\beta(t)\right]=0.
$
Thus, subtracting \eqref{phi-u1} with \eqref{phi-u-tau}, we derive
\begin{align}\label{Lphi-Lt}
\left|\E\left[\tau\mathcal{L}^{\Phi}\varphi(U^0)- \int_0^{\tau}\tilde{\mathcal{L}_t}\varphi(U^0) dt\right]\right|
\le\left|\E\left[\varphi(U(\tau))-\varphi(U^1)\right]\right|+C\tau^{2}.
\end{align}
Noticing that
\begin{align}
&\left|\int_0^{\tau}\E\left[\tilde{\mathcal{L}_t}\varphi(U^0) -\mathcal{L}\varphi(U^0)\right]dt\right|
\le\left|\int_0^{\tau}\E\left[D\varphi(U^0)\left(b(U(t))-b(U^0)\right)\right]dt\right|\nonumber\\\label{tilde}
&+\left|\frac12\int_0^{\tau}\E\left[D^2\varphi(U^0)\left(\sigma(U(t))-\sigma(U^0),\sigma(U(t))+\sigma(U^0)\right)\right]dt\right|
\end{align}
in which we have
\begin{align*}
&\left|\E\left[D\varphi(U^0)\left(b(U(t))-b(U^0)\right)\right]\right|
=\Big{|}\E\Big{[}D^2\varphi(U^0)\Big{(}\bi\frac1{h^2}A\left(U(t)-U^0\right)\\
&+\bi\lambda\Big{(}F(U(t))U(t)-F(U^0)U^0\Big{)}
-\hat E(U(t)-U^0)\Big{)}\Big{]}\Big{|}
\le C(th^{-2}+t)
\end{align*}
for the first term in \eqref{tilde}. In the last step, we have used the fact that $g(V):=F(V)V$, $\forall~V\in\mathcal{S}$ is a continuous differentiable function which satisfies $|D^k g(V)|\le C$ for $\|V\|\le1$ and $k\in\N$, and then replace $U(t)-U^0$ by the integral form of \eqref{spatial scheme} to get the result.
The second term in \eqref{tilde} can be estimated in the same way. 
Thus, we have 
\begin{align}\label{Lt-L}
\left|\int_0^{\tau}\E\left[\tilde{\mathcal{L}_t}\varphi(U^0) -\mathcal{L}\varphi(U^0)\right]dt\right|\le C(\tau^2h^{-2}+\tau^2).
\end{align}
We hence conclude based on \eqref{varphiU}, \eqref{Lphi-Lt}, \eqref{Lt-L} and Theorem \ref{weakerror} that
\begin{align}
\uppercase\expandafter{\romannumeral3}
=&\left|\frac1{N}\sum_{n=0}^{N-1}\E\left[\mathcal{L}^\Phi\varphi(U^n)-\mathcal{L}\varphi(U^n)+\frac1{6\tau}D^3\varphi(U^n)(U^{n+1}-U^n)^3\right]\right|\nonumber\\
\le&\frac1\tau\sup_{U^0\in\mathcal{S}}\left\{\left|\E\left[\tau\mathcal{L}^{\Phi}\varphi(U^0)- \int_0^{\tau}\tilde{\mathcal{L}_t}\varphi(U^0) dt\right]\right|
+\left|\int_0^{\tau}\E\left[\tilde{\mathcal{L}_t}\varphi(U^0) -\mathcal{L}\varphi(U^0)\right]dt\right|\right\}  \nonumber\\\label{term3}
&+C(\tau h^{-2}+\tau)\le C_h\tau.
\end{align}
Noticing that $\tau^3h^{-8}=O\left(\tau\right)$ under the condition $\tau=O(h^4)$, from \eqref{term1}, \eqref{term2} and \eqref{term3}, we finally obtain
\begin{align*}
\left|\mathbb{E}\left[\frac1N\sum_{n=0}^{N-1}f(U^n)-\hat f\right]\right|
\le C_h(\frac1T+\tau).
\end{align*}
\end{proof}

\begin{rk}\label{rk}
Based on the theorem above and the ergodicity of \eqref{spatial scheme}, for a fixed $h$, we obtain
\begin{align*}
\left|\mathbb{E}\left[\frac1N\sum_{n=0}^{N-1}f(U^n)-\frac1T\int_0^Tf(U(t))dt\right]\right|
\le C_h(B(T)+\tau),
\end{align*}
which implies that the global weak error is  of order one, i.e.,
\begin{align*}
\left|\mathbb{E}\Big{[}f(U^n)-f(U(t))\Big{]}\right|
\le C_h(\tilde B(t)+\tau),\quad t\in[n\tau,(n+1)\tau],
\end{align*}
where $B(T)\to0$ and $\tilde B(T)\to0$ as $T\to\infty$.
On the other hand, a time independent weak error in turn leads to the result stated in Theorem \ref{main}.
\end{rk}

\section{\textsc{Numerical experiments}}

In this section, numerical experiments are given to test several properties of scheme \eqref{full} with $\lambda=1$, i.e., the focusing case. In the following experiments, we simulate the noise $\delta_{n+1}\beta$ by $\sqrt{\tau}\xi_n$ with $\xi_n$ being independent $K$-dimensional $N(0,1)$-random variables, and choose $\eta_k=k^{-4}$, $k=1,\cdots,K$.
In addition, we approximate the expectation by taking averaged value over 500 paths, and the proposed scheme, which is implicit, is numerically solved utilizing the fixed point iteration.
In the sequel, we will use the notation $\|U\|_\gamma^\gamma:=\sum_{m=1}^M\left(|p_m|^\gamma+|q_m|^\gamma\right)$ for $U\in\C^M$ and $\gamma\in\N$ with $P=(p_1,\cdots,p_M)^T,Q=(q_1,\cdots,p_M)^T$ being the real and imaginary parts of $U$. Notice that $\|\cdot\|_2=\|\cdot\|$.

\begin{figure}[H]
\centering  
\subfigure[Proposed scheme]{
\begin{minipage}[t]{0.31\linewidth}
\includegraphics [width=1.7in]{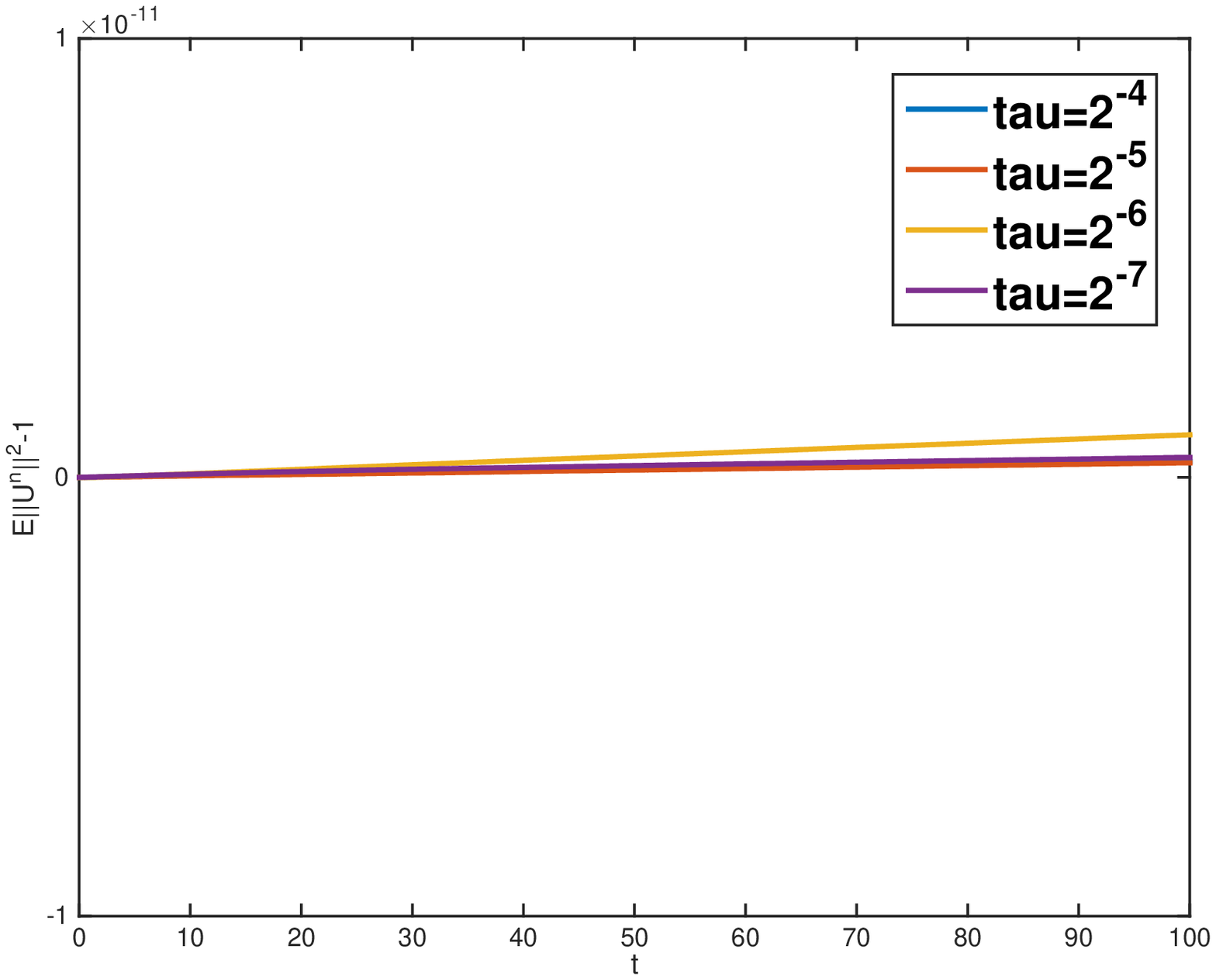}
\end{minipage}}
\subfigure[IME scheme]{
\begin{minipage}[t]{0.3\linewidth}
\includegraphics [width=1.7in]{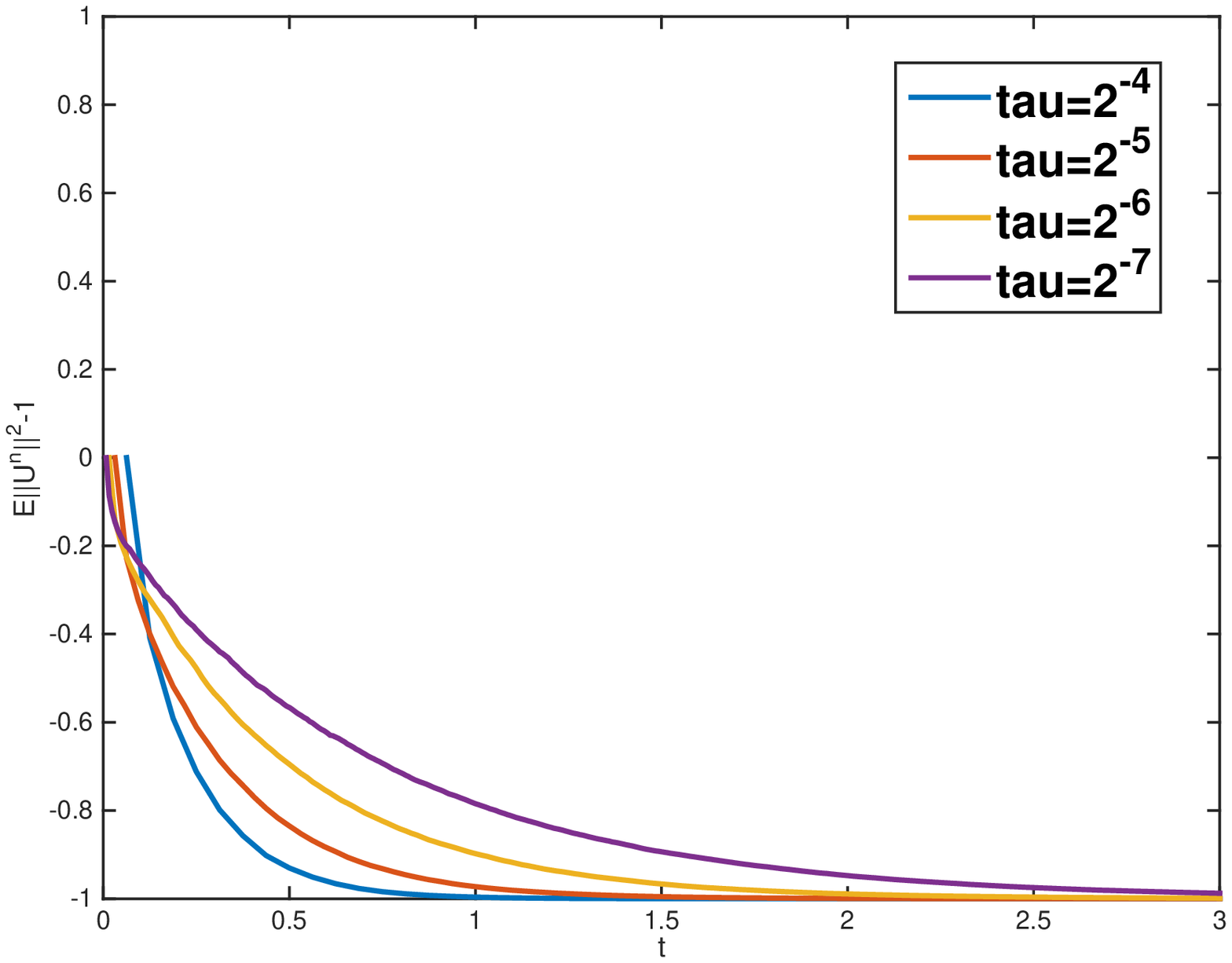}
\end{minipage}}
\subfigure[EM scheme]{
\begin{minipage}[t]{0.34\linewidth}
\includegraphics[width=1.7in]{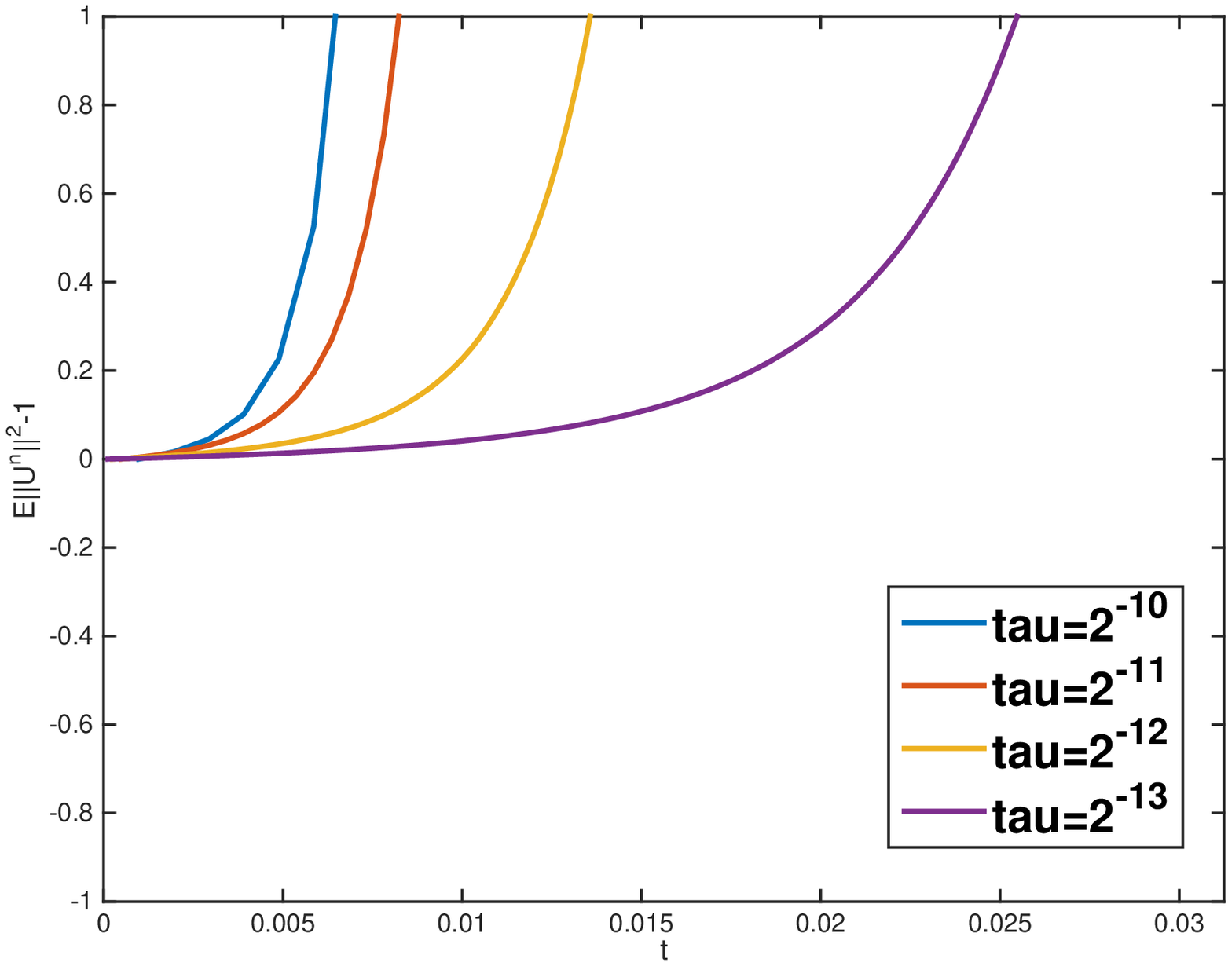}
\end{minipage}}
\caption{Charge evolution $\E\|U^n\|^2-1$ for (a) the proposed scheme with $T=100$ under steps $\tau=2^{-i}\,(i=4,5,6,7)$, (b) IME scheme with $T=3$ under steps $\tau=2^{-i}\,(i=4,5,6,7)$, and (c) EM scheme with $T=2^{-5}$ under steps $\tau=2^{-i}\,(i=10,11,12,13)$ $(h=0.05$, $K=30)$.}
\label{charge}
\end{figure}

As we omit the boundary nodes in the simulation, as a result, we may choose the normalized initial value $U^0=c_*(U^0(1),\cdots,U^0(M))^T$ based on function $u_0(x)$ satisfying $U^0(m)=u_0(mh)$, $m=1,\cdots,M$, in which $u_0(x)$ need not to satisfy the boundary condition in \eqref{model}. Let $u_0(x)=1$, and we get the normalized initial value $U^0$ satisfying $\|U^0\|=1$, which is used in Figures \ref{charge}, \ref{weakorder} and \ref{weakerrorex}. We first simulate the discrete charge for the proposed scheme compared with Euler--Maruyama (EM) scheme  and implicit Euler (IE) scheme, respectively. Figure \ref{charge} shows that the proposed scheme possesses the discrete charge conservation law $\E\|U^n\|^2=1$, which coincides with Proposition \ref{ju}, while both the EM scheme and the IE scheme do not. As the EM scheme does not stable, whose solution will blow up in a short time, we choose the time step $\tau$ small enough for the EM scheme in the experiments.

As the ergodic limit $\int_{\mathcal{S}}fd\mu_h$ is unknown, to verify the ergodicity of the numerical solution, we simulate the time averages $\frac1N\sum_{n=1}^N\E[f(U^n)]$ for the proposed scheme with the bounded function $f\in C_b(\mathcal{S})$ being (a) $f(U)=\|U\|_3^3$, (b) $f(U)=\sin(\|U\|_4^4)$ and (c) $f(U)=e^{-\|U\|_4^4}$ in Figure \ref{timeaver}, started from five different initial values $U^0_l,~1\le l\le 5$. It is known from Theorem \ref{ergo2} that for almost every initial values $U^0\in\mathcal{S}$, the time averages will converge to the same value, i.e. the ergodic limit. Thus, we choose five initial values 
$$U^0_l=c_*(U^0_l(1),\cdots,U^0_l(M))^T,~l=1,\cdots,5$$
based on the following five functions
\begin{align*}
u_{0,1}(x)=&\frac1{\sqrt{2}}+\frac{\bf{i}}{\sqrt{2}},\quad
u_{0,2}(x)=1,\quad
u_{0,3}(x)=2x,\\
u_{0,4}(x)=&\left(1-\sqrt{\frac{\pi}{2}(\exp{\frac14}-1)}\right)(1-\exp{(x(1-x))}),\nonumber\\
u_{0,5}(x)=&c_*\text{sech}(\frac{x}{\sqrt{2}})\exp{({\bf{i}} \frac{x}2)}\nonumber
\end{align*}
with $U_{l}^0(m)=u_{0,l}(hm),~1\le m\le M$ and $c_*$ being normalized constants.
The charge of all the initial functions equal one, and $u_{0,4}(x)$ even satisfies the boundary condition in \eqref{model}. 
Figure \ref{timeaver} shows that the proposed scheme started from different initial values converges to the same value with error no more than $O(\tau)$ with $h=0.05$ and $\tau=2^{-6}$, which coincides with Theorem \ref{main}.

\begin{figure}[H]
\centering  
\subfigure[$f(U)=\|U\|^3_3,\,T=20$]{
\begin{minipage}[t]{0.31\linewidth}
\includegraphics [width=1.7in]{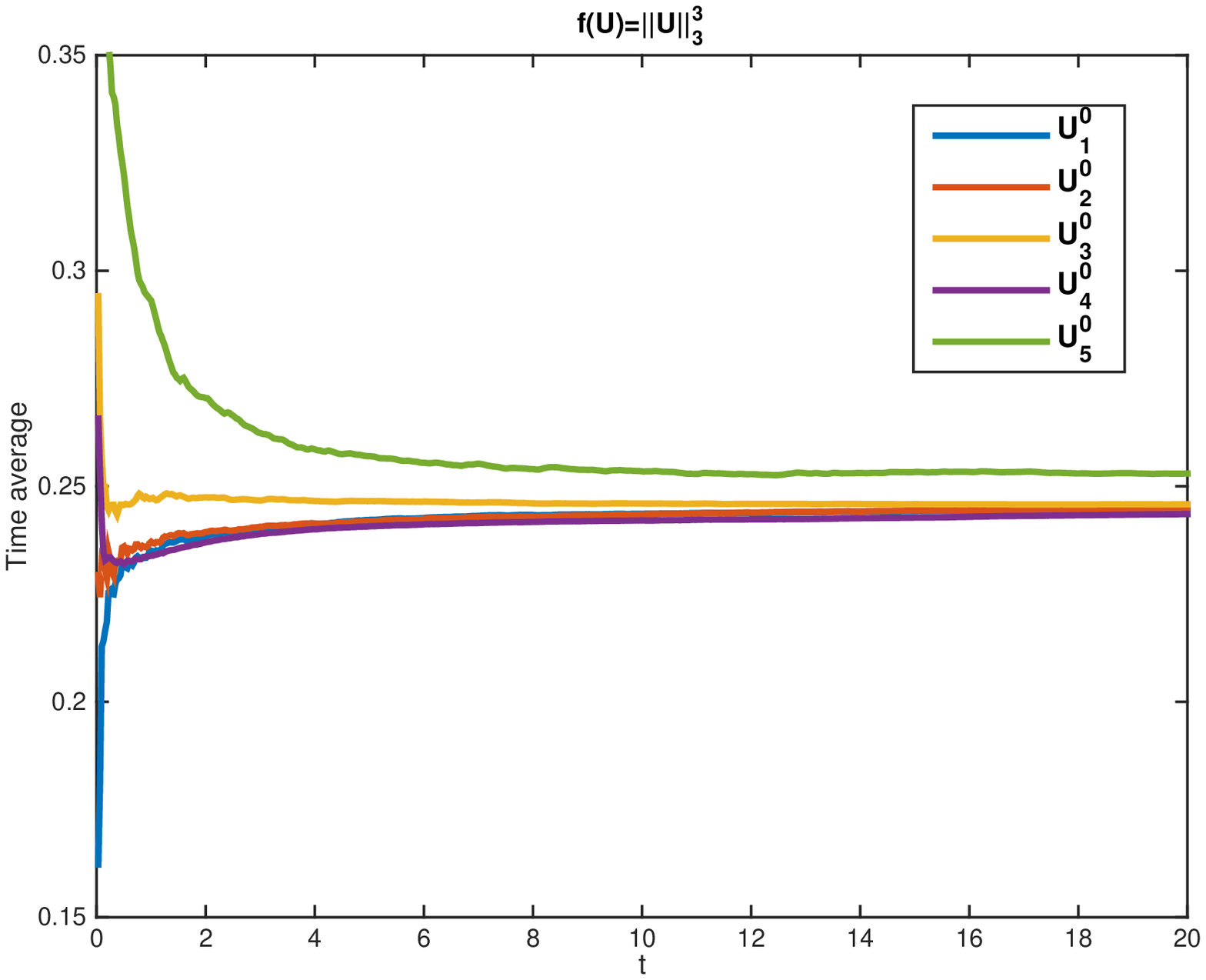}
\end{minipage}}
\subfigure[$f(U)=\sin(\|U\|_4^4),\,T=20$]{
\begin{minipage}[t]{0.3\linewidth}
\includegraphics [width=1.7in]{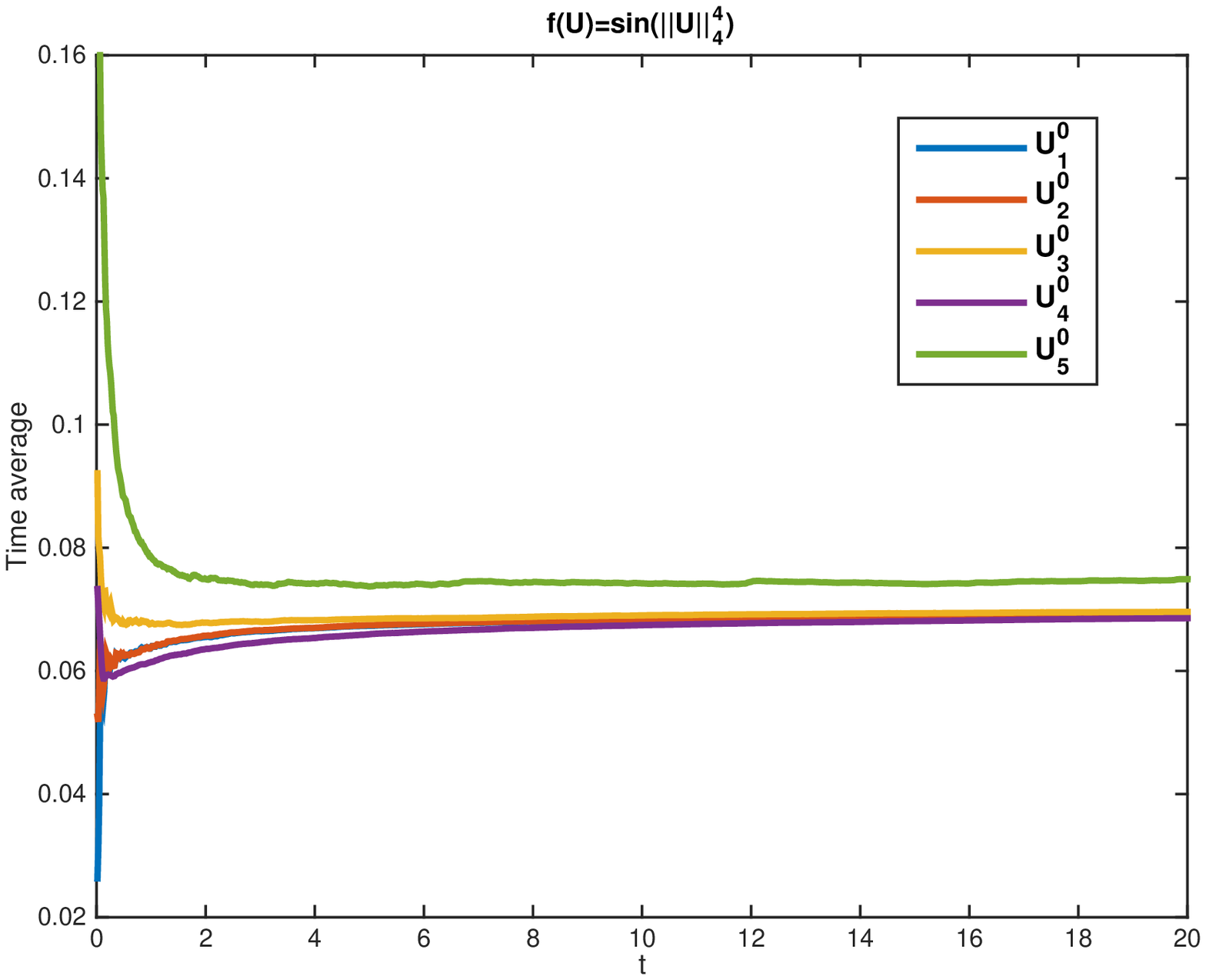}
\end{minipage}}
\subfigure[$f(U)=e^{-\|U\|_4^4},\,T=140$]{
\begin{minipage}[t]{0.34\linewidth}
\includegraphics[width=1.7in]{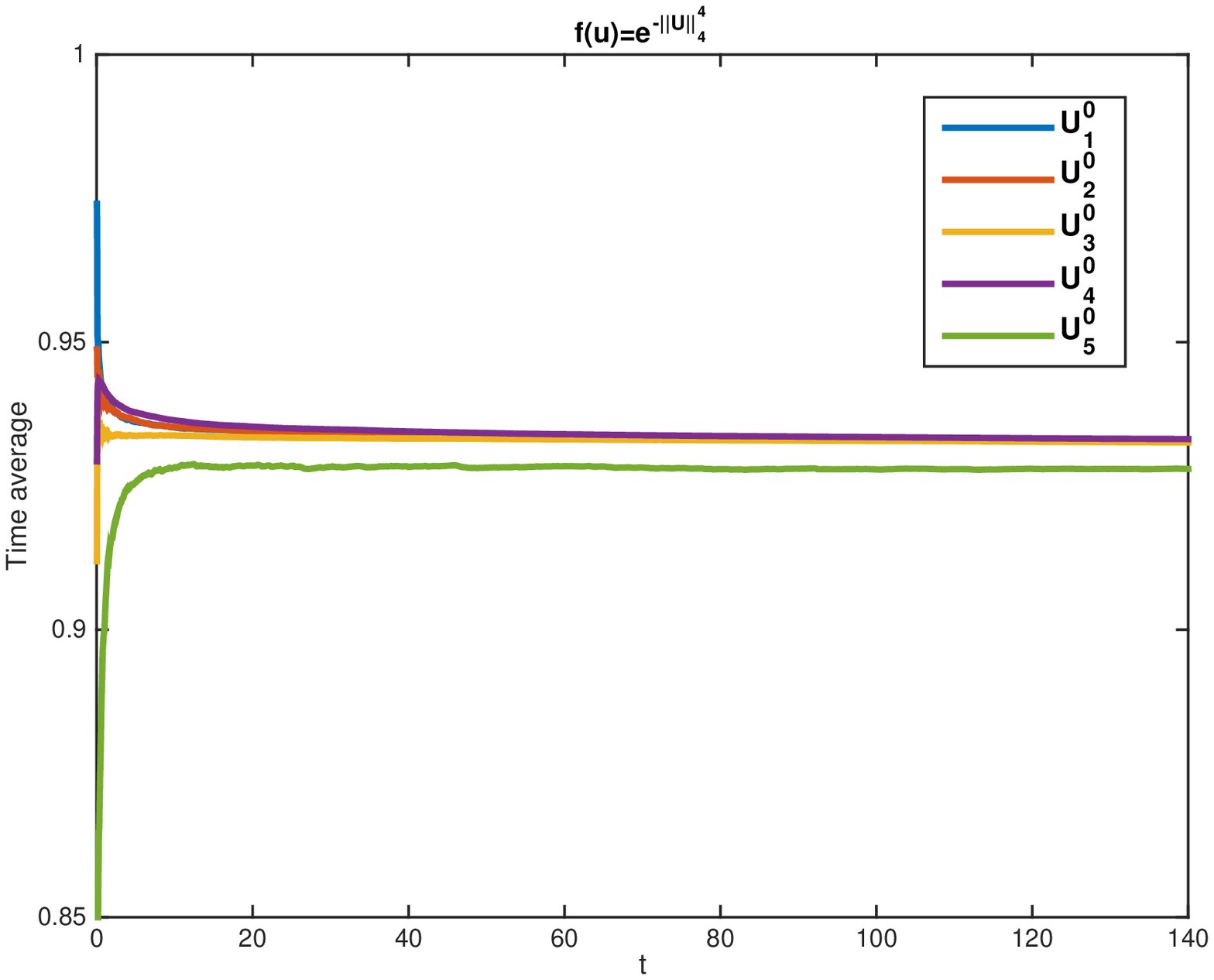}
\end{minipage}}
\caption{The time averages $\frac1N\sum_{n=1}^N\E[f(U^n)]$ for the proposed scheme with (a) $f(U)=\|U\|^3_3$, (b) $f(U)=\sin(\|U\|_4^4)$ and (c) $f(U)=e^{-\|U\|_4^4}$ $(\tau=2^{-6}$, $h=0.05$, $K=30)$.}
\label{timeaver}
\end{figure}
%%%%%%%

\begin{figure}[H]
\centering  
\subfigure[$f(U)=\|U\|^3_3$]{
\begin{minipage}[t]{0.31\linewidth}
\includegraphics [width=1.7in]{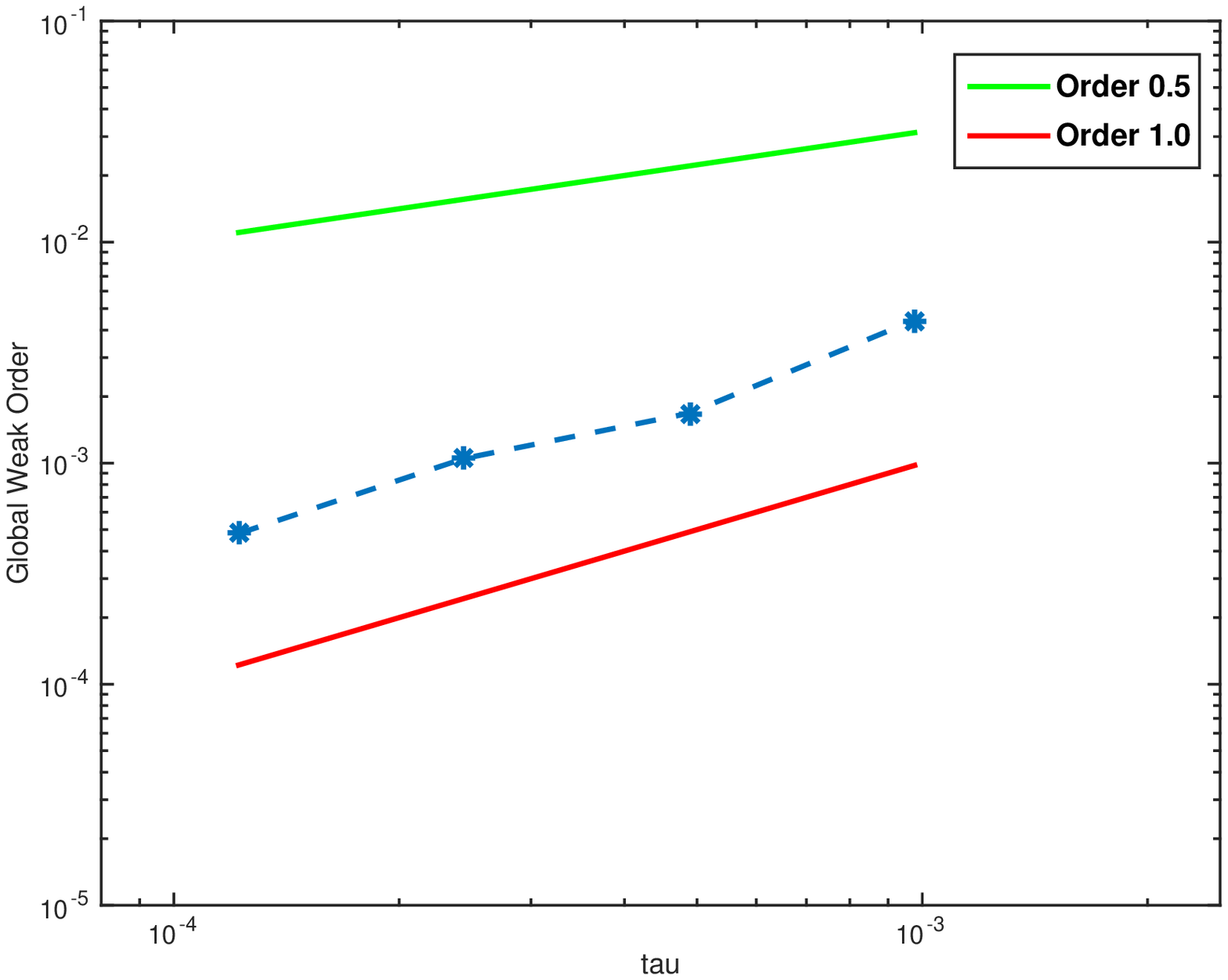}
\end{minipage}}
\subfigure[$f(U)=\sin(\|U\|_4^4)$]{
\begin{minipage}[t]{0.3\linewidth}
\includegraphics [width=1.7in]{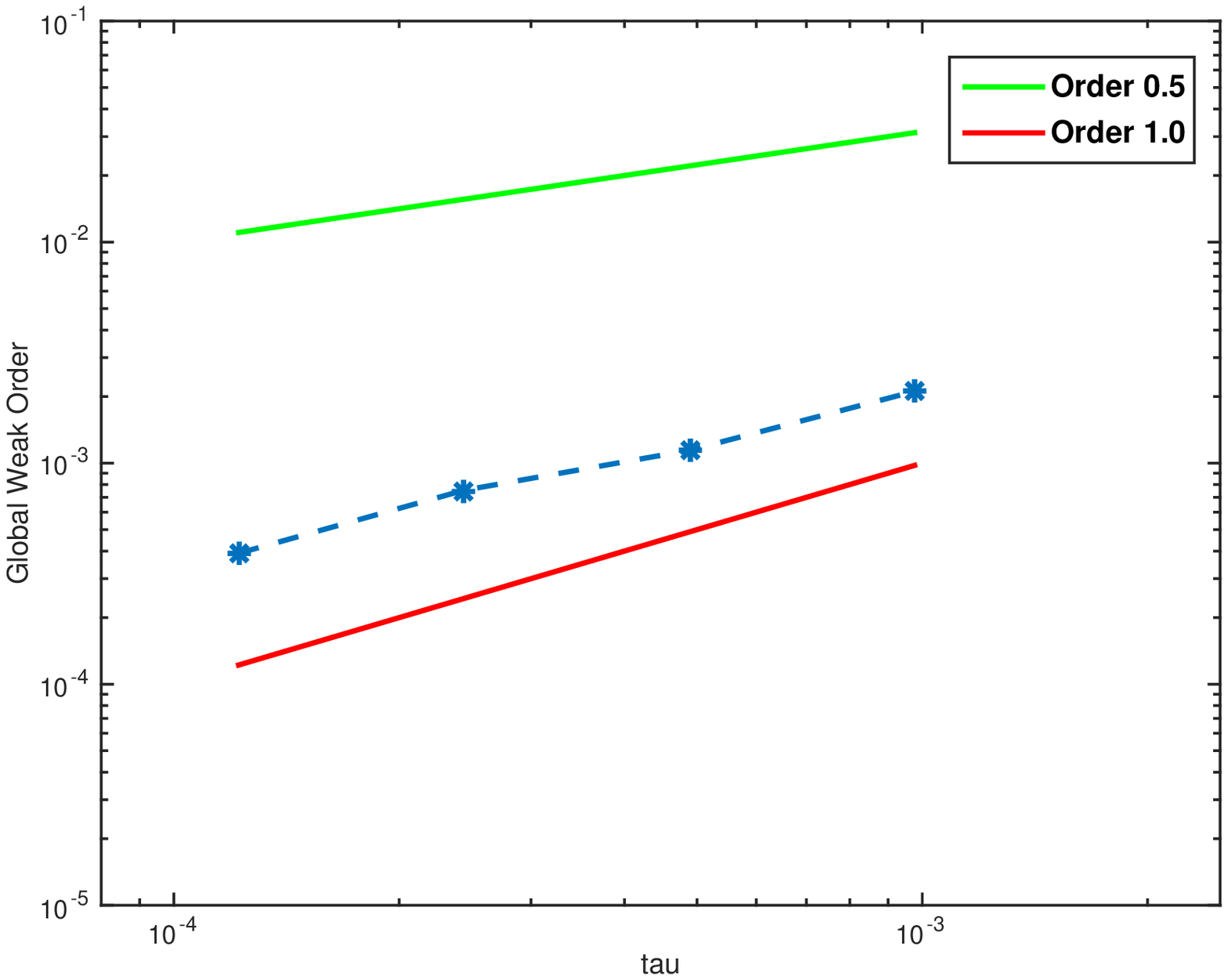}
\end{minipage}}
\subfigure[$f(U)=e^{-\|U\|_4^4}$]{
\begin{minipage}[t]{0.34\linewidth}
\includegraphics[width=1.7in]{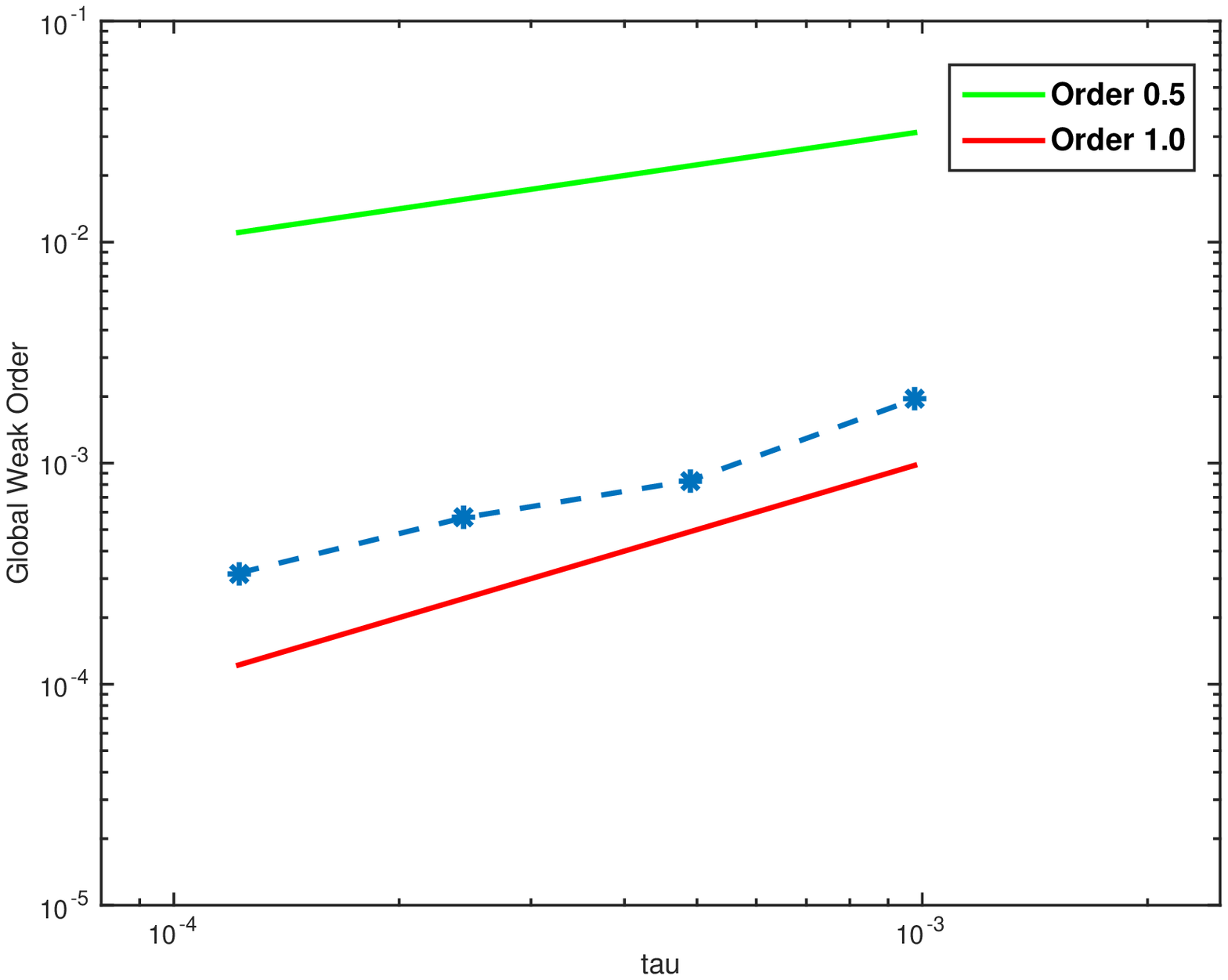}
\end{minipage}}
\caption{The weak convergence order of $|\E[f(U^n)-f(U(T))]|$ with (a) $f(U)=\|U\|_3^3$, (b) $f(U)=\sin(\|U\|_4^4)$ and (c) $f(U)=e^{-\|U\|_4^4}$ $(\tau=2^{-i},10\le i\le 13$, $h=0.05,$ $T=2^{-1},$ $K=30)$.}
\label{weakorder}
\end{figure}

For a fixed $h$, Figure \ref{weakorder} and \ref{weakerrorex} show the weak convergence order in temporal direction and the weak error over long time, respectively.
Figure \ref{weakorder} shows that the proposed scheme is of order one in the weak sense for (a) $f(U)=\|U\|_3^3$, (b) $f(U)=\sin(\|U\|_4^4)$ and (c) $f(U)=e^{-\|U\|_4^4}$ which coincides with the statement in Remark \ref{rk}.
Furthermore, based on the ergodicity for both FDS and FDA, the weak error is supposed to be independent of time interval when time is large enough. To verify this property, we simulate the weak error over long time in Figure \ref{weakerrorex} for (a) $f(U)=\|U\|_3^3$, (b) $f(U)=\sin(\|U\|_4^4)$ and (c) $f(U)=e^{-\|U\|_4^4}$, and it shows that the weak error for the proposed scheme would not increase before $T=1000$ while the weak error for the EM scheme would increase with time.

\begin{figure}[H]
\centering  
\subfigure[$f(U)=\|U\|^3_3$]{
\begin{minipage}[t]{0.31\linewidth}
\includegraphics [width=1.7in]{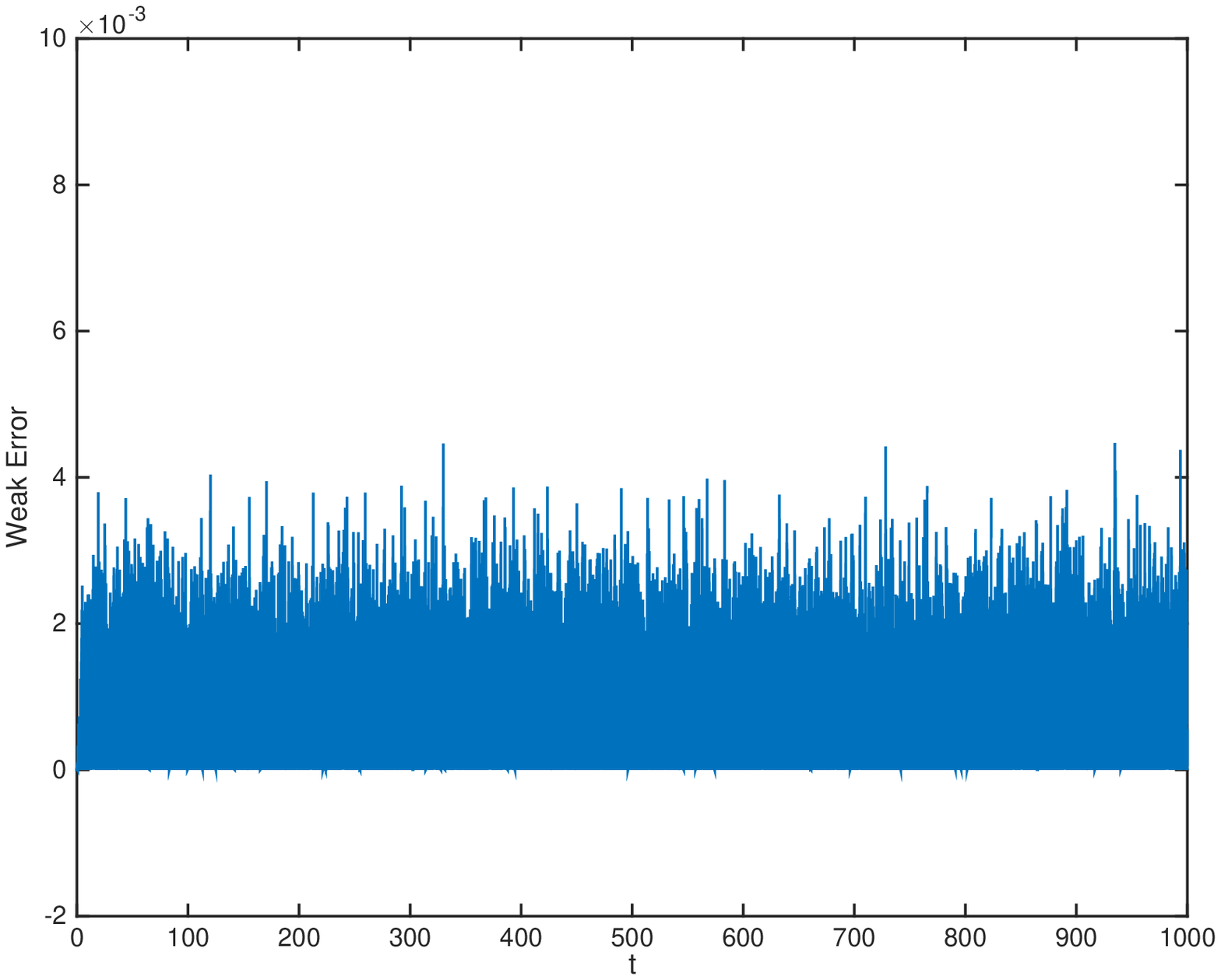}
\end{minipage}}
\subfigure[$f(U)=\sin(\|U\|_4^4)$]{
\begin{minipage}[t]{0.3\linewidth}
\includegraphics [width=1.7in]{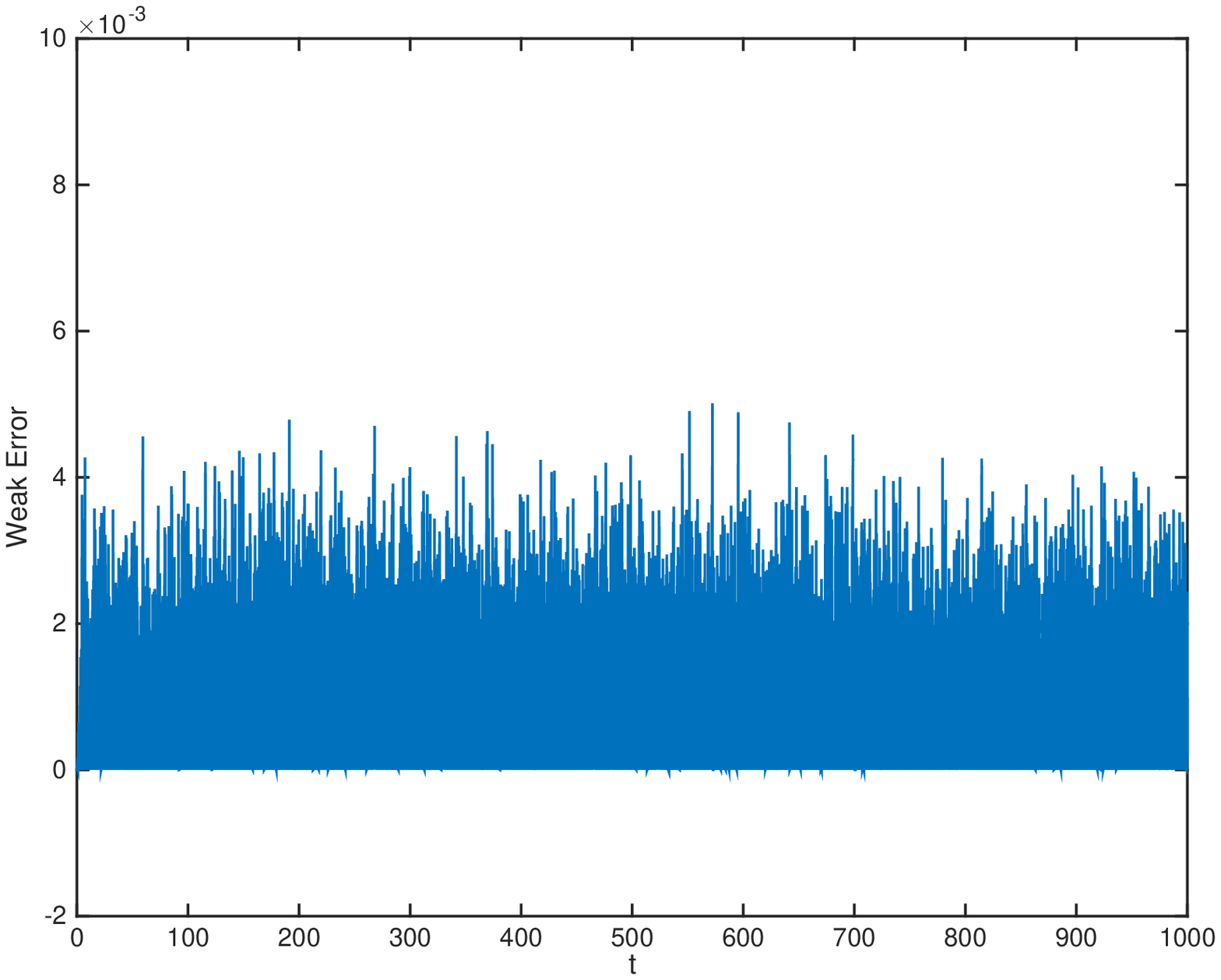}
\end{minipage}}
\subfigure[$f(U)=e^{-\|U\|_4^4}$]{
\begin{minipage}[t]{0.34\linewidth}
\includegraphics[width=1.7in]{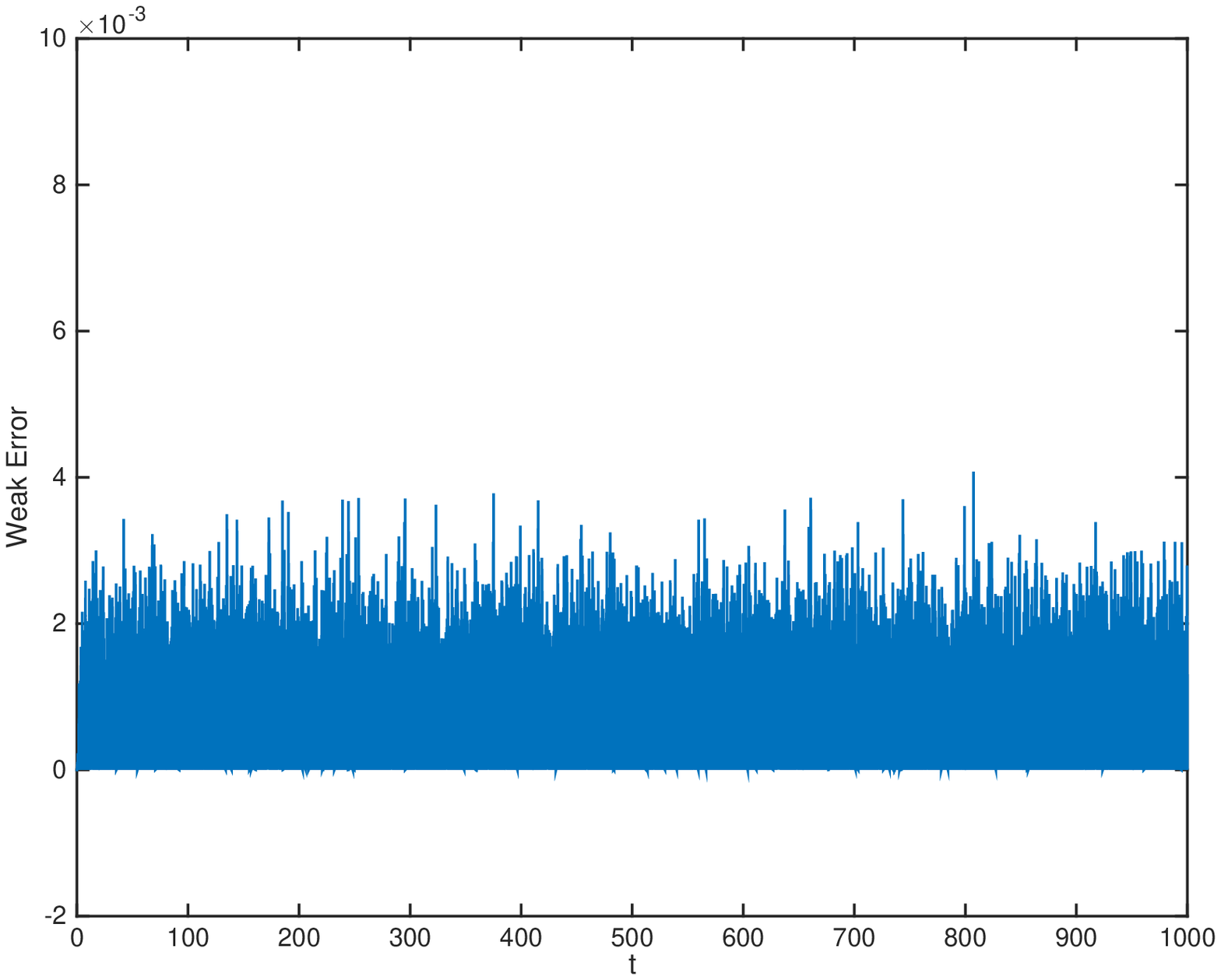}
\end{minipage}}
\caption{The weak error $|\E[f(U^n)-f(U(T))]|$ for (a) $f(U)=\|U\|_3^3$, (b) $f(U)=\sin(\|U\|_4^4)$ and (c) $f(U)=e^{-\|U\|_4^4}$ $(\tau=2^{-12}$, $h=0.05$, $T=10^3$, $K=30)$.}
\label{weakerrorex}
\end{figure}

\section{Appendix}

\subsection{Proof of Lemma \ref{stable}}
As proved in Proposition \ref{ju} that $\|U^{n}\|=1,\;\forall\;n\in\N$, for the nonlinear term, we have
\begin{align*}
\E\left\|F(U^{n+\frac12})U^{n+\frac12}\right\|^{2\gamma}
=\E\sum_{m=1}^M\left|u_m^{n+\frac12}\right|^{6\gamma}
\le\E\left(\sum_{m=1}^M\left|u_m^{n+\frac12}\right|^2\right)^{3\gamma}
\le\E\left\|U^{n+\frac12}\right\|^{6\gamma}
\le1
\end{align*}
by the convexity of $\mathcal{S}$, i.e., $\|U^{n+\frac12}\|\le1,$ a.s. The noise term can be estimated as
\begin{align}\label{Zbeta}
&\E\left\|Z(U^{n+\frac12})\delta_{n+1}\beta\right\|^{2\gamma}
=\E\left(\sum_{m=1}^M\left|\sum_{k=1}^Ku_m^{n+\frac12}e_k(x_m)\sqrt{\eta_k}\delta_{n+1}\beta_k\right|^2\right)^{\gamma}\nonumber\\
\le&
\E\left(2\sum_{m=1}^M\left|u_m^{n+\frac12}\right|^{2}
\Big{(}\sum_{k=1}^K\sqrt{\eta_k}|\delta_{n+1}\beta_k|\Big{)}^{2}\right)^\gamma
=\E\left(2\left\|U^{n+\frac12}\right\|^{2}
\Big{(}\sum_{k=1}^K\sqrt{\eta_k}|\delta_{n+1}\beta_k|\Big{)}^{2}\right)^\gamma\nonumber\\
\le& C\E\left(\sum_{k=1}^K\eta_k^\frac14\eta_k^\frac14|\delta_{n+1}\beta_k|\right)^{2\gamma}
\le C\E\bigg{[}\Big{(}\sum_{k=1}^K\eta_k^{\frac{\gamma}{2(2\gamma-1)}}\Big{)}^{2\gamma-1}\Big{(}\sum_{k=1}^K\eta_k^{\frac{\gamma}2}|\delta_{n+1}\beta_k|^{2\gamma}\Big{)}\bigg{]}
\le C\tau^\gamma
\end{align}
by $|e_k(x_m)|\le\sqrt{2}$ and H\"older's inequality. 
In the last step of \eqref{Zbeta}, noticing that, as $Q\in\mathcal{HS}(L^2,H^{\frac32-\frac1\gamma})$, that is, $\sum_{k=1}^{\infty}k^{3-\frac2\gamma}\eta_k<\infty$, so $\eta_k=O(k^{-(4-\frac2\gamma+\epsilon)})$ for any $\epsilon>0$. 
Thus, 
$$\sum_{k=1}^{\infty}\eta_k^{\frac{\gamma}{2(2\gamma-1)}}\le C\sum_{k=1}^{\infty}k^{-(4-\frac2\gamma+\epsilon)\frac{\gamma}{2(2\gamma-1)}}=C\sum_{k=1}^{\infty}k^{-\left(1+\frac{\epsilon\gamma}{2(2\gamma-1)}\right)}<\infty.$$
In conclusion,
\begin{align*}
&\E\left\|U^{n+1}-U^n\right\|^{2\gamma}\\
\le&C\left(\E\left\|\frac{\tau}{h^2}AU^{n+\frac12}\right\|^{2\gamma}
+\E\left\|\lambda\tau F(U^{n+\frac12})U^{n+\frac12}\right\|^{2\gamma}
+\E\left\|Z(U^{n+\frac12})\delta_{n+1}\beta\right\|^{2\gamma}\right)\\
\le&\frac{C\tau^{2\gamma}}{h^{4\gamma}}\E\left\|U^{n+\frac12}\right\|^{2\gamma}
+C\tau^{2\gamma}
+C\tau^\gamma
\le C\left(\tau^{2\gamma}h^{-4\gamma}+\tau^\gamma\right),
\end{align*}
where we have used the fact that $\|A\|\le4$.

\subsection{Proof of Lemma \ref{stable2}}
From \eqref{spatial scheme} and \eqref{noise}, based on H\"older's inequality, we obtain
\begin{align*}
&\E\|U(t_{n+1})-U(t_{n})\|^{2\gamma}\\
=&\E\left\|\int_{t_{n}}^{t_{n+1}}\left[\mathbf{i}\frac1{h^2}AU+\mathbf{i}\lambda F(U)U-\hat{E}U\right]dt+\int_{t_{n}}^{t_{n+1}}\mathbf{i}Z(U)d\beta(t)\right\|^{2\gamma}\\
\le&C\Bigg{(}\int_{t_{n}}^{t_{n+1}}\E\left\|\mathbf{i}\frac1{h^2}AU+\mathbf{i}\lambda F(U)U-\hat{E}U\right\|^{2\gamma}dt\left(\int_{t_{n}}^{t_{n+1}}1^{\frac{2\gamma}{2\gamma-1}}dt\right)^{2\gamma-1}\\
&+\E\left\|\int_{t_{n}}^{t_{n+1}}\mathbf{i}Z(U)d\beta(t)\right\|^{2\gamma}\Bigg{)}\\
\le&C\tau^{2\gamma-1}\left\|\frac1{h^2}A\right\|^{2\gamma}\int_{t_{n}}^{t_{n+1}}\E\left\|U\right\|^{2\gamma}dt
+C{\tau}^{2\gamma}
+C{\tau^\gamma}\\
\le&C({\tau}^{2\gamma}h^{-{4\gamma}}+\tau^\gamma),
\end{align*}
where we have used the boundedness of $F(U)U$ in $\mathcal{S}$ similar to that in Lemma \ref{stable}. In the third step of the equation above, we also used
\begin{align*}
&\E\|\hat{E}U\|^{2\gamma}\le C\E\left(\sum_{m=1}^M\left|\sum_{k=1}^K\eta_ke_k^2(x_m)u_m\right|^2\right)^\gamma\\
\le& C\E\left(\sum_{m=1}^M|u_m|^2\left(\sum_{k=1}^K\eta_k\right)^2\right)^\gamma
\le C\eta^{2\gamma}\E\|U\|^{2\gamma}
\le C
\end{align*}
and
\begin{align*}
&\E\left\|\int_{t_{n}}^{t_{n+1}}\mathbf{i}Z(U)d\beta(t)\right\|^{2\gamma}
\le C\left(\int_{t_n}^{t_{n+1}}\left(\E\|Z(U)\|_{\mathcal{HS}}^{2\gamma}\right)^{\frac1\gamma}dt\right)^\gamma\\
&\le C\left(\int_{t_n}^{t_{n+1}}\left(\E\left(\sum_{m=1}^M\sum_{k=1}^K\left|u_me_k(x_m)\sqrt{\eta_k}\right|^2\right)^\gamma\right)^{\frac1\gamma}dt\right)^\gamma\\
&\le C\left(\int_{t_n}^{t_{n+1}}\left(\E\left(2\eta\|U\|^2\right)^\gamma\right)^{\frac1\gamma}dt\right)^\gamma
\le C\tau^\gamma
\end{align*}
according to the Burkholder--Davis--Gundy inequality and the fact that the Hilbert--Schmidt operater norm $\|Z(U)\|_{\mathcal{HS}}=\|Z(U)\|_F$ with $\|\cdot\|_F$ denoting the Frobenius norm.

\subsection{Proof of Theorem \ref{weakerror}}
Based on Taylor expansion, Lemma \ref{stable} and \ref{stable2}, we obtain
\begin{align*}
&\E\left[\varphi(U(\tau))-\varphi(U^1)\right]
=\E\left[D\varphi(U^1)\big{(}U(\tau)-U^1\big{)}
+O\big{(}\|U(\tau)-U^1\|^2\big{)}\right]\\
=&\E\left[D\varphi(U^0)\big{(}U(\tau)-U^1\big{)}\right]+\E\left[D^2\varphi(U^0)(U^1-U^0,U(\tau)-U^1)\right]\\
&+O\Big{(}\E\left[\|U^1-U^0\|^2\|U(\tau)-U^1\|\right]+\E\|U(\tau)-U^1\|^2\Big{)}\\
=:&\mathcal{A}+\mathcal{B}+\mathcal{C}.
\end{align*}
We give the mild solution and discrete mild solution of  \eqref{spatial scheme} and \eqref{full} respectively,
\begin{align*}
U(\tau)=&e^{\mathbf{i}\frac{1}{h^2}A\tau}U^0
+\int_0^\tau e^{\mathbf{i}\frac{1}{h^2}A(\tau-s)}\left(\mathbf{i}\lambda F(U(s))U(s)-\hat EU(s)\right)ds\\
&+\int_0^\tau e^{\mathbf{i}\frac{1}{h^2}A(\tau-s)}\bi Z(U(s))d\beta(s),
\end{align*}
\begin{align*}
U^1=&(I-\frac{\mathbf{i}\tau}{2h^2}A)^{-1}(I+\frac{\mathbf{i}\tau}{2h^2}A)U^0
+(I-\frac{\mathbf{i}\tau}{2h^2}A)^{-1}\bi\lambda\tau F\left(U^{\frac12}\right)U^{\frac12}\\
&+(I-\frac{\mathbf{i}\tau}{2h^2}A)^{-1}\bi Z\left(U^{\frac12}\right)\delta_{1}\beta.
\end{align*}

{\bf Estimation of $\mathcal{A}$.} Considering the difference between above equations, we have
\begin{align*}
U(\tau)-U^1
=&\left(e^{\mathbf{i}\frac{1}{h^2}A\tau}-(I-\frac{\mathbf{i}\tau}{2h^2}A)^{-1}(I+\frac{\mathbf{i}\tau}{2h^2}A)\right)U^0\nonumber\\
&+\bi\int_0^\tau\left[e^{\bi\frac{1}{h^2}A(\tau-s)}-(I-\frac{\mathbf{i}\tau}{2h^2}A)^{-1}\right]\lambda F(U(s))U(s)ds\nonumber\\
&+\bi\int_0^\tau(I-\frac{\mathbf{i}\tau}{2h^2}A)^{-1}\lambda \left[F(U(s))U(s)-F\left(U^{\frac12}\right)U^{\frac12}\right]ds\nonumber\\
&+\bi\int_0^\tau\left[e^{\bi\frac{1}{h^2}A(\tau-s)}-\left(I-\frac{\bi\tau}{2h^2}A\right)^{-1}\right]
Z(U(s))d\beta(s)\nonumber\\
&+\bi\int_0^\tau\left(I-\frac{\bi\tau}{2h^2}A\right)^{-1}
Z(U(s)-U^0)d\beta(s)\nonumber\\
&-\left[\frac{\bi}2\left(I-\frac{\bi\tau}{2h^2}A\right)^{-1}
Z(U^1-U^0)\delta_{1}\beta
+\int_0^\tau e^{\bi\frac{1}{h^2}A(\tau-s)}\hat EU(s)ds\right],\nonumber\\
=:&\mathbf{a}+\mathbf{b}+\mathbf{c}+\mathbf{d}+\mathbf{e}+\mathbf{f},
\end{align*}
which, together with the fact that $\E[D\varphi(U^0)\mathbf{d}]=\E[D\varphi(U^0)\mathbf{e}]=0$, 
yields that
\begin{align*}
\mathcal{A}=&\E\left[D\varphi(U^0)\mathbf{a}\right]+\E\left[D\varphi(U^0)\mathbf{b}\right]+\E\left[D\varphi(U^0)\mathbf{c}\right]+\E\left[D\varphi(U^0)\mathbf{f}\right]\\
=:&A_1+A_2+A_3+A_4.
\end{align*}
Based on the estimates $e^x-(1-\frac{x}2)^{-1}(1+\frac{x}2)=O(x^3)$ for $\|x\|<1$, and
\begin{align}\label{operator}
\left\|e^{\bi\frac{1}{h^2}A(\tau-s)}-(I-\frac{\mathbf{i}\tau}{2h^2}A)^{-1}\right\|\le C\left(\frac{\tau}{h^2}\left\|A\right\|\right)\le C\tau h^{-2},\quad\forall~ s\in[0,\tau],
\end{align}
we have
\begin{align}\label{a1}
|A_1|\le C\|\varphi\|_{1,\infty}\|\tau h^{-2}A\|^3\E\|U^0\|\le C\tau^3h^{-6}\le C\tau^2h^{-2}
\end{align}
under the condition $\tau=O(h^4)$, and 
\begin{align}\label{a2}
|A_2|\le C\|\varphi\|_{1,\infty}\int_0^{\tau}\|\tau h^{-2}A\|\|F(U(s))U(s)\|ds\le C\tau^2h^{-2}.
\end{align}
Term $A_3$ can be estimated based on Lemma \ref{stable} and \ref{stable2}.
\begin{align*}
|A_3|=&\Bigg{|}\E\Bigg{[}D\varphi(U^0)\int_0^\tau(I-\frac{\mathbf{i}\tau}{2h^2}A)^{-1} \bigg{[}\left(F(U(s))U(s)-F(U^0)U^0\right)\nonumber\\
&-\left(F\left(U^{\frac12}\right)U^{\frac12}-F(U^0)U^0\right)\bigg{]}ds\Bigg{]}\Bigg{|}
\end{align*}
in which we have known from the proof of Theorem \ref{main} that
\begin{align*}
&F(U(s))U(s)-F(U^0)U^0=g(U(s))-g(U^0)\\
=&Dg(U^0)(U(s)-U^0)+\frac12D^2g(\theta(s))(U(s)-U^0)^2\\
=&Dg(U^0)\left(\int_0^s\frac{\bi}{h^2}AU(r)+\bi\lambda F(U(r))U(r)-\hat EU(r)dr+\int_0^sZ(U(r))d\beta(r)\right)\\
&+\frac12D^2g(\theta(s))(U(s)-U^0)^2
\end{align*}
for some $\theta(s)\in[U^0,U(s)]$ and $s\in[0,\tau]$, and the same for the term $F\left(U^{\frac12}\right)U^{\frac12}-F(U^0)U^0$.
Based on the fact that $\E\left[Dg(U^0)\int_0^sZ(U(r))d\beta(r)\right]=0$, we hence get
\begin{align}\label{a3}
|A_3|\le C(\tau^2h^{-2}+\tau^2)
\end{align}
similar to the proof of Lemma \ref{stable2}. 
Rewrite
\begin{align*}
Z(U^1-U^0)\delta_{1}\beta=& 
\begin{pmatrix}
u_1^1-u_1^0& & \\
 &\ddots& \\
  & &u_M^1-u_M^0
\end{pmatrix}E_{MK}\Lambda\delta_{1}\beta\\
=&\begin{pmatrix}
\sum_{k=1}^Ke_k(x_1)\sqrt{\eta_k}\delta_1\beta_k& & \\
 &\ddots& \\
  & &\sum_{k=1}^Ke_k(x_M)\sqrt{\eta_k}\delta_1\beta_k
\end{pmatrix}(U^1-U^0)\\
=:&G(U^1-U^0),
\end{align*}
where $G$ satisfies that
$
\E[GU^0]=0.
$
Utilizing that $\E[GF(U^0)U^0]=0$, we can rewrite term $A_4$ as
\begin{align*}
A_4=&-\E\Bigg{[}D\varphi(U^0)\left(\frac{\bi}2\left(I-\frac{\bi\tau}{2h^2}A\right)^{-1}G(U^1-U^0)
+\int_0^\tau e^{\bi\frac{1}{h^2}A(\tau-s)}\hat EU(s)ds\right)\Bigg{]}\\
=&-\frac{\bi}2\E\Bigg{[}D\varphi(U^0)\left(I-\frac{\bi\tau}{2h^2}A\right)^{-1}
G\Bigg{(}\bi\frac{\tau}{h^2}AU^{\frac12}
+\bi\lambda\tau F(U^{\frac12})U^{\frac12}
+\bi GU^{\frac12}\Bigg{)}\Bigg{]}\\
&-\E\Bigg{[}D\varphi(U^0)\int_0^\tau e^{\bi\frac{1}{h^2}A(\tau-s)}\hat EU(s)ds\Bigg{]}\\
=&\frac{\tau}{4h^2}\E\left[D\varphi(U^0)\left(I-\frac{\bi\tau}{2h^2}A\right)^{-1}GA(U^1-U^0)\right]\\
&+\frac12\lambda\tau\E\left[D\varphi(U^0)\left(I-\frac{\bi\tau}{2h^2}A\right)^{-1}G\left(F(U^{\frac12})U^{\frac12}-F(U^{0})U^{0}\right)\right]\\
&+\frac14\E\left[D\varphi(U^0)\left(I-\frac{\bi\tau}{2h^2}A\right)^{-1}
G^2(U^1-U^0)\right]\\
&+\E\left[D\varphi(U^0)\left(\left(I-\frac{\bi\tau}{2h^2}A\right)^{-1}
\frac12G^2U^0
-\int_0^\tau e^{\bi\frac{1}{h^2}A(\tau-s)}\hat EU(s)ds\right)\right]\\
=:&A_{4,1}+A_{4,2}+A_{4,3}+A_{4,4},
\end{align*}
in which, based on $\E[G^3U^0]=0$, $A_{4,3}$ can be expressed as
\begin{align*}
\frac14\E\left[D\varphi(U^0)\left(I-\frac{\bi\tau}{2h^2}A\right)^{-1}G^2
\left(\bi\frac{\tau}{h^2}AU^{\frac12}+\bi\tau\lambda F(U^{\frac12})U^\frac12+\frac{\bi}2G(U^1-U^0)\right)\right].
\end{align*}
For any $U\in\C^M$, we have 
\begin{align*}
\E\|GU\|=\E\|Z(U)\delta_1\beta\|\le C\E\left(\|U\|^2\Big{(}\sum_{k=1}^K\sqrt{\eta_k}|\delta_1\beta_k|\Big{)}^2\right)^{\frac12}
\le C\tau^\frac12\left(\E\|U\|^2\right)^\frac12.
\end{align*}
Hence $\E\|G^3(U^1-U^0)\|\le C\tau^{\frac12}(\E\|G^2(U^1-U^0)\|^2)^\frac12$ can be further estimated based on \eqref{Zbeta} with $\gamma=4$ under the condition $Q\in\mathcal{HS}(L^2,H^\frac54)$, 
which together with Lemma \ref{stable} and $\|U^{\frac12}\|\leq1$ yields
\begin{align}\label{a4123}
|A_{4,1}+A_{4,2}+A_{4,3}|\le C(\tau^{\frac52}h^{-4}+\tau^{2}h^{-2}+\tau^{2})\le C(\tau^{2}h^{-2}+\tau^{2}).
\end{align}
For the term $A_{4,4}$, we have
\begin{align*}
\frac12G^2U^0
\overset{\E}{=}\frac12\begin{pmatrix}
\sum_{k=1}^Ke_k^2(x_1)\eta_k(\delta_1\beta_k)^2u_1^0\\
\vdots\\
\sum_{k=1}^Ke_k^2(x_M)\eta_k(\delta_1\beta_k)^2u_M^0
\end{pmatrix},~
\hat EU(s)=\frac12\begin{pmatrix}
\sum_{k=1}^Ke_k^2(x_1)\eta_ku_1(s)\\
\vdots\\
\sum_{k=1}^Ke_k^2(x_M)\eta_ku_M(s)
\end{pmatrix}.
\end{align*}
Thus, we obtain 
\begin{align}
A_{4,4}=&\frac12\E\left[D\varphi(U^0)\left(I-\frac{\bi\tau}{2h^2}A\right)^{-1}
\begin{pmatrix}
\sum_{k=1}^Ke_k^2(x_1)\eta_k(\delta_1\beta_k)^2u_1^0\\
\vdots\\
\sum_{k=1}^Ke_k^2(x_M)\eta_k(\delta_1\beta_k)^2u_M^0
\end{pmatrix}\right]\nonumber\\
-&\frac12\E\left[D\varphi(U^0)\int_0^\tau e^{\bi\frac{1}{h^2}A(\tau-s)}\begin{pmatrix}
\sum_{k=1}^Ke_k^2(x_1)\eta_ku_1(s)\\
\vdots\\
\sum_{k=1}^Ke_k^2(x_M)\eta_ku_M(s)
\end{pmatrix}ds\right]\nonumber\\
=&\frac12\E\left[D\varphi(U^0)\left(I-\frac{\bi\tau}{2h^2}A\right)^{-1}
\begin{pmatrix}
\sum_{k=1}^Ke_k^2(x_1)\eta_k((\delta_1\beta_k)^2-\tau)u_1^0\\
\vdots\\
\sum_{k=1}^Ke_k^2(x_M)\eta_k((\delta_1\beta_k)^2-\tau)u_M^0
\end{pmatrix}\right]\nonumber\\
+&\frac12\E\left[D\varphi(U^0)\int_0^\tau \left(\left(I-\frac{\bi\tau}{2h^2}A\right)^{-1}-e^{\bi\frac{1}{h^2}A(\tau-s)}\right)
\begin{pmatrix}
\sum_{k=1}^Ke_k^2(x_1)\eta_ku_1^0\\
\vdots\\
\sum_{k=1}^Ke_k^2(x_M)\eta_ku_M^0
\end{pmatrix}ds\right]\nonumber\\\label{a441}
-&\frac12\E\left[D\varphi(U^0)\int_0^\tau e^{\bi\frac{1}{h^2}A(\tau-s)}\begin{pmatrix}
\sum_{k=1}^Ke_k^2(x_1)\eta_k\left(u_1(s)-u_1^0\right)\\
\vdots\\
\sum_{k=1}^Ke_k^2(x_M)\eta_k\left(u_M(s)-u_M^0\right)
\end{pmatrix}ds\right],
\end{align}
where in the last step we have used the fact
\begin{align*}
\begin{pmatrix}
\sum_{k=1}^Ke_k^2(x_1)\eta_k\tau u_1^0\\
\vdots\\
\sum_{k=1}^Ke_k^2(x_M)\eta_k\tau u_M^0
\end{pmatrix}=
\int_0^\tau
\begin{pmatrix}
\sum_{k=1}^Ke_k^2(x_1)\eta_ku_1^0\\
\vdots\\
\sum_{k=1}^Ke_k^2(x_M)\eta_ku_M^0
\end{pmatrix}ds.
\end{align*}
Noticing that the first term in \eqref{a441} vanishes as $\E(\delta_1\beta_k)^2=\tau$ and replacing $U(s)-U^0$ by the integral type of \eqref{spatial scheme}, then further calculation shows that
\begin{align}\label{a44}
|A_{4,4}|\le C(\tau^{2}h^{-2}+\tau^{2})
\end{align}
based on \eqref{operator} and the technique used in \eqref{a3}.
We then conclude from \eqref{a1}--\eqref{a44} that
\begin{align}\label{A}
|\mathcal{A}|\le C\big{(}\tau^2h^{-2}+\tau^{2}\big{)}\le C_h\tau^2.
\end{align}

{\bf Estimation of $\mathcal{C}$.} Estimations of $A_1$ and $A_2$ show that
\begin{align}\label{err1}
\E\|\mathbf{a}+\mathbf{b}\|^2\le C\big{(}\tau^{6}h^{-{12}}+\tau^4h^{-4}\big{)}\le C\tau^3.
\end{align}
Based on H\"older's inequality, It\^o isometry, Lemma \ref{stable} and \ref{stable2}, we have
\begin{align}\label{err2}
\E\|\mathbf{c}+\mathbf{d}\|^2\le C\tau\int_0^\tau\E\|U(s)-U^{\frac12}\|^2ds+\int_0^\tau C\tau^2h^{-4}ds\le C(\tau^3h^{-4}+\tau^3)
\end{align}
and
\begin{align}\label{err3}
\E\|\mathbf{e}\|^2\le C\E\left[\int_0^\tau\left\|\left(I-\frac{\bi\tau}{2h^2}A\right)^{-1}
Z\left(U(s)-U^0\right)\right\|_{\mathcal{HS}}^2
 ds\right]
 \le C\tau^2. 
\end{align}
Rewriting $Z(U^1-U^0)\delta_1\beta=G\left(\bi\frac{\tau}{h^2}AU^{\frac12}
+\bi\lambda\tau F(U^{\frac12})U^{\frac12}
+\bi GU^{\frac12}\right)$, which together with the H\"older's inequality and \eqref{Zbeta} yields
\begin{align}\label{err4}
\E\|\mathbf{f}\|^2\le C(\tau^3h^{-4}+\tau^2).
\end{align}
We then conclude from \eqref{err1}--\eqref{err4} and the condition $\tau=O(h^4)$ that
\begin{align}\label{utau-u1}
\E\|U(\tau)-U^1\|^2\le C\tau^2,
\end{align}
which yields
\begin{align}\label{C}
|\mathcal{C}|=O\left(\left(\E\|U^1-U^0\|^4\right)^{\frac12}\left(\E\|U(\tau)-U^1\|^2\right)^{\frac12}+\E\|U(\tau)-U^1\|^2\right)\le C\tau^{2}.
\end{align}

{\bf Estimation of $\mathcal{B}$.} As for
$\mathcal{B}=\E\left[D^2\varphi(U^0)\left(U^1-U^0,
\mathbf{a}+\mathbf{b}+\mathbf{c}+\mathbf{d}+\mathbf{e}+\mathbf{f}\right)\right]$,
according to the H\"older's inequality, \eqref{err1} and \eqref{err2}, 
we have
\begin{align*}
&\left|\E\left[D^2\varphi(U^0)\left(U^1-U^0,
\mathbf{a}+\mathbf{b}+\mathbf{c}+\mathbf{d}\right)\right]\right|\\
\le &C\left(\E\|U^1-U^0\|^2\right)^\frac12\left(\E\|\mathbf{a}+\mathbf{b}+\mathbf{c}+\mathbf{d}\|^2\right)^\frac12
\le C(\tau^2h^{-2}+\tau^2).
\end{align*} 
Noticing that
\begin{align*}
&\E\left[D^2\varphi(U^0)\left(U^1-U^0,\mathbf{e}+\mathbf{f}\right)\right]\\
=&\E\left[D^2\varphi(U^0)\left(U^1-U^0,\bi\int_0^\tau\left(I-\frac{\bi\tau}{2h^2}A\right)^{-1}
Z(U(s)-U^1)d\beta(s)\right)\right]\nonumber\\
&+\frac12\E\left[D^2\varphi(U^0)\left(U^1-U^0,\bi\left(I-\frac{\bi\tau}{2h^2}A\right)^{-1}
Z(U^1-U^0)\delta_{1}\beta\right)\right]\\
&-\E\left[D^2\varphi(U^0)\left(U^1-U^0,\int_0^\tau e^{\bi\frac{1}{h^2}A(\tau-s)}\hat EU(s)ds\right)\right]\\
=:&B_1+B_2+B_3,
\end{align*}
where $|B_1|\le C\tau^2$ according to \eqref{utau-u1} and Lemma \ref{stable}.
Furthermore, 
\begin{align*}
B_2=&\frac12\E\bigg{[}D^2\varphi(U^0)\Big{(}\bi\frac{\tau}{h^2}AU^{\frac12}+\bi\tau\lambda F(U^{\frac12})U^{\frac12},\bi\left(I-\frac{\bi\tau}{2h^2}A\right)^{-1}
Z(U^1-U^0)\delta_{1}\beta\Big{)}\bigg{]}\\
&+\frac12\E\bigg{[}D^2\varphi(U^0)\Big{(}\bi Z\left(\frac{U^1-U^0}2\right)\delta_1\beta,\bi\left(I-\frac{\bi\tau}{2h^2}A\right)^{-1}
Z(U^1-U^0)\delta_{1}\beta\Big{)}\bigg{]}\\
&+\frac12\E\bigg{[}D^2\varphi(U^0)\Big{(}\bi Z(U^0)\delta_1\beta,\bi\left(I-\frac{\bi\tau}{2h^2}A\right)^{-1}
Z(U^1-U^0)\delta_{1}\beta\Big{)}\bigg{]}\\
=:&B_{2,1}+B_{2,2}+B_{2,3}
\end{align*}
with $|B_{2,1}+B_{2,2}|\le C(\tau^2h^{-2}+\tau^2)$. Replacing $U^1-U^0$ again by \eqref{full}, we obtain
\begin{align*}
|B_{2,3}|\le&\left|\frac12\E\bigg{[}D^2\varphi(U^0)\Big{(}\bi Z(U^0)\delta_1\beta,\bi\left(I-\frac{\bi\tau}{2h^2}A\right)^{-1}
Z\left(\bi Z(U^\frac12)\delta_{1}\beta\right)\delta_{1}\beta\Big{)}\bigg{]}\right|\\
&+C(\tau^2h^{-2}+\tau^2)\\
\le&\left|\frac12\E\bigg{[}D^2\varphi(U^0)\Big{(}\bi Z(U^0)\delta_1\beta,\bi\left(I-\frac{\bi\tau}{2h^2}A\right)^{-1}
Z\left(\bi Z(U^0)\delta_{1}\beta\right)\delta_{1}\beta\Big{)}\bigg{]}\right|\\
&+C(\tau^2h^{-2}+\tau^2)\\
=&C(\tau^2h^{-2}+\tau^2),
\end{align*}
where in the last step we used the fact $\E[(\delta_1\beta)^3]=0$ and $U^0$ is $\mathcal{F}_0$-adapted. Also,
\begin{align*}
|B_3|\le&\left|\E\left[D^2\varphi(U^0)\left(\bi\frac{\tau}{h^2}AU^{\frac12}+\bi\tau\lambda F(U^{\frac12})U^\frac12,\int_0^\tau e^{\bi\frac{1}{h^2}A(\tau-s)}\hat EU(s)ds\right)\right]\right|\\
&+\left|\E\left[D^2\varphi(U^0)\left(\bi Z(U^\frac12)\delta_1\beta,\int_0^\tau e^{\bi\frac{1}{h^2}A(\tau-s)}\hat E\left(U(s)-U^0\right)ds\right)\right]\right|\\
&+\left|\E\left[D^2\varphi(U^0)\left(\bi Z(U^\frac12)\delta_1\beta,\int_0^\tau e^{\bi\frac{1}{h^2}A(\tau-s)}\hat EU^0ds\right)\right]\right|\\
\le&C(\tau^2h^{-2}+\tau^2)+\frac12\left|\E\left[D^2\varphi(U^0)\left(\bi Z(U^1-U^0)\delta_1\beta,\int_0^\tau e^{\bi\frac{1}{h^2}A(\tau-s)}\hat EU^0ds\right)\right]\right|\\
\le&C(\tau^2h^{-2}+\tau^2),
\end{align*}
so we finally obtain
\begin{align*}
|\mathcal{B}|\le C(\tau^2h^{-2}+\tau^2)\le C_h\tau^2,
\end{align*}
which, together with \eqref{A}, \eqref{C}, completes the proof.

\section*{Acknowledgement}

Authors are grateful to Prof. Zhenxin Liu for helpful suggestions and discussions.

%\nocite{*} %%%all the references in bib appear
\bibliography{wangxu}
\bibliographystyle{plain}

\end{document}